\documentclass[11pt,reqno,a4paper]{amsart}
\usepackage{hyperref}
\usepackage{enumerate}
\usepackage[left=0.65in, right=0.65in, top= 1in, bottom=1.75in]{geometry}
\usepackage{amsmath,amssymb,amsthm,amsfonts}
\usepackage{graphicx}
\usepackage{tikz}
\usetikzlibrary{graphs, graphs.standard}

\tikzset{
	modal/.style={>=stealth’,shorten >=1pt,shorten <=1pt,auto,node distance=1.5cm,
		semithick},
	world/.style={circle, draw,minimum size=.1cm,fill=gray!15},
	point/.style={circle,draw,inner sep=0.3mm,fill=black},
	circ/.style={circle,draw,inner sep=0.1mm,fill=white},
	reflexive above/.style={->,loop,looseness=7,in=120,out=60},
	reflexive below/.style={->,loop,looseness=7,in=240,out=300},
	reflexive left/.style={->,loop,looseness=7,in=150,out=210},
	reflexive right/.style={->,loop,looseness=7,in=30,out=330}
}

\usetikzlibrary{shapes}
\usetikzlibrary{plotmarks}
\usetikzlibrary{arrows}
\usetikzlibrary{positioning}
\theoremstyle{definition}
\newtheorem{defn}{Definition}[section]
\newtheorem{fact}[defn]{Fact}

\newtheorem{thm}[defn]{Theorem}

\newtheorem{lem}[defn]{Lemma}
\newtheorem{question}[defn]{Question}
\newtheorem{remark}[defn]{Remark}
\newtheorem{observation}[defn]{Observation}
\newtheorem{claim}[defn]{Claim}
\title[Brooks' type theorems and K\H{o}nig's Lemma]{Brooks' type theorems 
for coloring parameters of locally finite graphs and K\H{o}nig's Lemma}

\author{Amitayu Banerjee$^{\ast}$}
\address{Alfr\'ed R\'enyi Institute of Mathematics, Reáltanoda utca 13-15, 1053, Budapest, Hungary}
\email[Corresponding author]{banerjee.amitayu@gmail.com}
\thanks{$^{\ast}$ Corresponding author.}

\author{Zal\'{a}n Moln\'{a}r}
\address{E\"otv\"os Lor\'and University, Department of Logic, M\'{u}zeum krt. 4, 1088, Budapest, Hungary}
\email{mozaag@gmail.com}

\author{Alexa Gopaulsingh}
\address{E\"otv\"os Lor\'and University, Department of Logic, M\'{u}zeum krt. 4, 1088, Budapest, Hungary}
\email{alexa279e@gmail.com}

\date{}

\makeatletter
\@namedef{subjclassname@2020}{\textup{2020} Mathematics Subject Classification}
\makeatother
\begin{document}
\begin{abstract}
In the past, analogues to Brooks’ theorem have been found for various parameters of graph coloring for infinite locally finite connected graphs in $\mathsf{ZFC}$. We prove these theorems are not provable in $\mathsf{ZF}$ (i.e., the Zermelo–Fraenkel set theory without the Axiom of Choice ($\mathsf{AC}$)).
Moreover, such theorems follow from K\H{o}nig's Lemma (every infinite locally finite connected graph has a ray--a weak form of $\mathsf{AC}$) in $\mathsf{ZF}$.
In $\mathsf{ZF}$, inspired by a combinatorial argument of Herrlich and Tachtsis from 2006, we formulate new conditions for the existence of the distinguishing chromatic number, the distinguishing chromatic index, the total chromatic number, the total distinguishing chromatic number, the odd chromatic number, and the neighbor-distinguishing index in infinite locally finite connected graphs, which are equivalent to
K\H{o}nig’s Lemma. In this direction, we strengthen a recent result of Stawiski from 2023.
We also generalize an algorithm of Imrich, Kalinowski, Pil\'{s}niak, and Shekarriz to show that the statement ``If $G$ is a connected infinite graph where the maximum degree $\Delta(G) \geq 3$ is finite, then the list-distinguishing chromatic number is at most $2\Delta(G)-1$'' holds under K\H{o}nig's Lemma in $\mathsf{ZF}$. However, the above statement fails in $\mathsf{ZF}$. 
\end{abstract}
\maketitle

\section{Introduction}
In 1941, Brooks \cite{Bro1941} proved that for any connected undirected finite graph $G$ with maximum degree $\Delta(G)$, the chromatic number $\chi(G)$ of $G$ is at most $\Delta(G)$, unless $G$ is a complete graph or an odd cycle. Brooks' theorem for connected infinite graphs was proved in 1976 by Behzad and Radjavi \cite{BR1976}. Since then, several authors have found general upper bounds for various coloring parameters of an infinite graph $G$ (cf. \cite{BR1976, BP2015, IKT2007, IKPS2017, LPS2022, PS2020}). 
For the reader’s convenience, complete definitions of the parameters for graph coloring are stated in the Appendix, so that the proofs appear earlier.

\begin{fact}\label{Fact 1.1} ($\mathsf{ZFC}$)
{\em If $G$ is a connected infinite graph where $\Delta(G)$ is finite, then the
following holds:

\begin{enumerate}
    \item {(Behzad and Radjavi; \cite[Theorem 3]{BR1976})} The chromatic number $\chi(G)$ is at most $\Delta(G)$.
    
    \item {(Behzad and Radjavi; \cite[Theorem 7]{BR1976})} The chromatic index $\chi'(G)$ is at most $\Delta(G)+1$.
    
    \item {(Imrich, Klav\v{z}ar, and Trofimov; \cite[Theorem 2.1]{IKT2007})} The distinguishing number $D(G)$ is at most $\Delta(G)$.

    \item {(Broere and Pil\'{s}niak; \cite[Theorem 4]{BP2015})} The distinguishing index $D'(G)$ is at most $\Delta(G)$.

    \item {(Imrich, Kalinowski, Pil\'{s}niak, and Shekarriz; \cite[Theorem 3]{IKPS2017})} The distinguishing chromatic number $\chi_{D}(G)$ is at most $2\Delta(G)-1$.
    
    \item {(Imrich, Kalinowski, Pil\'{s}niak, and Shekarriz; \cite[Theorem 14]{IKPS2017})} The distinguishing chromatic index $\chi'_{D}(G)$ is at most $\Delta(G)+2$.

    \item {(Lehner, Pil\'{s}niak, and Stawiski; \cite[Theorem 3]{LPS2022})} If $\Delta(G)\geq 3$, then $D(G)\leq\Delta(G)-1$.

    \item {(Pil\'{s}niak and Stawiski; \cite[Theorem 3]{PS2020})} If $\Delta(G)\geq 3$, then $D'(G)\leq \Delta(G)-1$. 

    \item  {(H\"{u}ning, Imrich, Kloas, Schreiber, and Tucker; \cite[Corollary 2]{HIKST2019})} If $\Delta(G)=3$, then $D(G)\leq 2$.
\end{enumerate}
}
\end{fact}

The coloring parameters can be drastically influenced by the inclusion or exclusion of $\mathsf{AC}$ (cf. \cite{BMG2024, GK1991, HT2006, HR2005, Soi2005, SS2004, SS2003, Sta2023}).
Galvin--Komj\'{a}th \cite{GK1991} proved that $\mathsf{AC}$ is equivalent to the statement ``Every graph has a chromatic number". 
In a series of papers, Shelah and Soifer \cite{Soi2005, SS2004, SS2003} have constructed some graphs whose chromatic number is finite in $\mathsf{ZFC}$ and uncountable (if it exists) in some models of $\mathsf{ZF}$ (e.g. Solovay's model from \cite[Theorem 1]{Sol1970}).
Herrlich--Tachtsis \cite[Proposition 23]{HT2006} proved that no Russell graph has a chromatic number in $\mathsf{ZF}$.\footnote{The reader is referred to Herrlich–Tachtsis \cite{HT2006} for the details concerning Russell graph.} 
The above results are the main motivation for us to investigate the status of the results in Fact \ref{Fact 1.1} when $\mathsf{AC}$ is removed.

\subsection{Results in $\mathsf{ZFC}$:}
Analogues to Brooks' theorems for the total chromatic number \cite{MR1998}, the total distinguishing number \cite[Theorem 2.2]{KPW2016}, the total distinguishing chromatic number \cite[Theorem 4.2]{KPW2016}, and the odd chromatic number \cite[Theorem 4.9]{CPS2022} were investigated for finite graphs (cf. Fact \ref{Fact 3.1}). 
Ferrara et al. \cite{FGHSW2013} extended the notion of distinguishing proper coloring to a list distinguishing proper coloring and introduced the notion of list distinguishing chromatic number for finite graphs. 
We prove the following:

\begin{enumerate}
    \item (Theorem \ref{Theorem 3.3}) In $\mathsf{ZF}$, K\H{o}nig’s Lemma implies the statement {\em ``If $G$ is a connected infinite graph where $\Delta(G)\geq 3$ is finite, then 
    \begin{enumerate}
        \item the total distinguishing number $D''(G)$ is at most $\lceil \sqrt{\Delta(G)}\rceil$,
        \item the odd chromatic number $\chi_{o}(G)$ is at most $2\Delta(G)$,
        \item the total chromatic number $\chi''(G)$ is at most $\Delta(G)+10^{26}$, and
        \item the total distinguishing chromatic number $\chi''_{D}(G)$ is at most $\Delta(G)+10^{26}+1$''.
    \end{enumerate}
}  

The results mentioned in (1) are {\em new results in $\mathsf{ZFC}$} concerning the estimation of the upper bound of the listed coloring parameters for infinite
locally finite connected graphs.
These results are extensions of the known results from Fact \ref{Fact 3.1} to infinite locally finite connected graphs.

    \item (Theorem \ref{Theorem 7.4}) In $\mathsf{ZF}$, K\H{o}nig’s Lemma implies the statement {\em ``If $G$ is a connected infinite graph where $\Delta(G)\geq 3$ is finite, then the list-distinguishing chromatic number $\chi_{D_{L}}(G)$ is at most $2\Delta(G)-1$''.
}  

Obtaining the upper bounds for the list-distinguishing chromatic number of infinite graphs were not previously investigated in $\mathsf{ZFC}$. Thus, it seems that the result mentioned in (2) is a {\em new result in $\mathsf{ZFC}$} concerning the estimation of the upper bound of the list-distinguishing chromatic number for infinite locally finite connected graphs. Inspired by the arguments of Imrich, Kalinowski, Pil\'{s}niak, and Shekarriz \cite[Theorem 3]{IKPS2017}, we use the methods in this note and generalize the algorithm of \cite[Theorem 3]{IKPS2017} to prove the statement of (2).

\item In Remark \ref{Remark 7.3}, we show that the inequality ``$\chi_{D}(G)< \chi_{D_{L}}(G)$" holds in general for a connected infinite graph $G$ where $\Delta(G)\geq 3$ is finite in $\mathsf{ZFC}$. 
    Thus, the result mentioned in the statement of (2) strengthens Fact \ref{Fact 1.1}(5) (cf. Theorem \ref{Theorem 7.4}).\footnote{The result mentioned in (1) also provides new upper bounds for the corresponding coloring parameters
for connected infinite graphs in $\mathsf{ZFC}$. Theorems \ref{Theorem 3.3}, \ref{Theorem 4.1}, \ref{Theorem 5.1}, \ref{Theorem 6.3}, \ref{Theorem 7.4} are due to the first two authors. Moreover, Theorem \ref{Theorem 6.1}, and Remark \ref{Remark 7.3} are due to all the authors.}
\end{enumerate}

\subsection{Consistency results in $\mathsf{ZF}$:}

\begin{enumerate}[{(A)}]
\item (Theorems \ref{Theorem 6.1}, \ref{Theorem 6.3}, \ref{Theorem 7.4}) There are models of $\mathsf{ZF}$ where the statements in Fact \ref{Fact 1.1}(1)-(9), (1), and (2) fail.

\item Stawiski \cite{Sta2023} proved that the statement 
``Every infinite connected subcubic graph $G$ (i.e. $\Delta(G)=3$) has a distinguishing number''
fails in some model of $\mathsf{ZF}$ if one assumes that the sets of colors can be well-ordered  (cf. \cite[Lemma 3.3 and section 2]{Sta2023}). 
In Theorem \ref{Theorem 6.3}, we 
strengthen Stawiski's result.
In particular, we show that Stawiski's result holds even if the sets of colors cannot be well-ordered.     
\end{enumerate}     

\subsection{New equivalents of K\H{o}nig's Lemma:}
Inspired by a combinatorial argument of Herrlich–Tachtsis \cite[Proposition 23]{HT2006}, the first two authors investigate the role of K\H{o}nig’s Lemma in the existence of the total chromatic number, and some recently studied coloring parameters such as the odd chromatic number \cite{CPS2022}, the distinguishing chromatic number \cite{IKPS2017}, the distinguishing chromatic index \cite{IKPS2017}, the total distinguishing chromatic number \cite{IKPS2017}, and the neighbor-distinguishing index \cite{HCWH2023}. The following table summarizes the new results (cf. Theorems \ref{Theorem 4.1}, \ref{Theorem 5.1}) as well as the published results \cite{Ban2023, BMG2024, DM2006, Sta2023} on this topic.\footnote{K\H{o}nig's Lemma has several known graph-theoretic equivalents. Fix any integer $m \geq 4$. Delhomm\'{e}--Morillon \cite{DM2006}, Banerjee \cite{Ban2023, Ban}, Banerjee--Moln\'{a}r--Gopaulsingh \cite[Theorems 4.2, 5.1]{BMG2024} and Stawiski \cite[Theorem 3.8]{Sta2023} proved that K\H{o}nig’s lemma is equivalent to the statement ``Every infinite locally finite connected graph has a chromatic number ({\em resp.} a chromatic index, a distinguishing number, a distinguishing index, an irreducible proper coloring, a spanning tree, a spanning $m$-bush, a maximal independent sets, a minimal edge cover, a maximal matching, a minimal dominating set)".}
In the following table, $\mathcal{P}_{\text{lf,c}}$(property $X$) denotes ``Every infinite locally finite connected graph has property $X$".

\begin{center}
\begin{tabular}{|l|l|l| } 
			\hline Known equivalents of K\H{o}nig's lemma & New equivalents of K\H{o}nig's lemma \\ 
                \hline\hline
                $\mathcal{P}_{\text{lf,c}}$(distinguishing number) & $\mathcal{P}_{\text{lf,c}}$(distinguishing chromatic number) \\ 
			$\mathcal{P}_{\text{lf,c}}$(distinguishing index)  & $\mathcal{P}_{\text{lf,c}}$(distinguishing chromatic index)\\
                $\mathcal{P}_{\text{lf,c}}$(chromatic number)  & $\mathcal{P}_{\text{lf,c}}$(total chromatic number)\\
            $\mathcal{P}_{\text{lf,c}}$(chromatic index)  & $\mathcal{P}_{\text{lf,c}}$(total distinguishing chromatic number)\\
                $\mathcal{P}_{\text{lf,c}}$(irreducible proper coloring)  & $\mathcal{P}_{\text{lf,c}}$(neighbor-distinguishing number) \\

            $\mathcal{P}_{\text{lf,c}}$(spanning tree)  & $\mathcal{P}_{\text{lf,c}}$(neighbor-distinguishing index) \\
                
			$\mathcal{P}_{\text{lf,c}}$(maximal independent set)   & $\mathcal{P}_{\text{lf,c}}$(odd chromatic number) \\    
            \hline 
            
		\end{tabular}
\end{center}


\subsection{Remarks and Questions}
In Remark \ref{Remark 8.1}, we observe that the statements (1)-(6), (9) in Fact \ref{Fact 1.1} follow from K\H{o}nig's Lemma in $\mathsf{ZF}$. We also add some open questions of interest.
\section{Basics}
\begin{defn}\label{Definition 2.1}
Let $X$ and $Y$ be two sets. We write:	
\begin{enumerate}
	\item $X \preceq Y$, if there is an injection $f : X \rightarrow Y$.
	\item $X$ and $Y$ are equipotent if $X \preceq Y$ and $Y \preceq X$, i.e., if there is a bijection $f : X \rightarrow Y$.
	\item $X \prec Y$, if $X \preceq Y$ and $X$ is not equipotent with $Y$.
\end{enumerate}
\end{defn}
 \begin{defn}\label{Definition 2.2}
The von Neumann hierarchy is defined by recursion on the ordinals:
\begin{enumerate}
    \item $V_{0}=\emptyset$,
    \item $V_{\alpha +1}=\mathcal{P}(V_{\alpha})$ where $\mathcal{P}(X)$ is the powerset of a set $X$,
    \item $V_{\alpha}=\bigcup_{\beta<\alpha}V_{\beta}$ if $\alpha$ is a limit ordinal.
\end{enumerate}
The von Neumann universe $V$ is defined as $\bigcup_{\alpha \in Ord}{V_{\alpha}}$ where $Ord$ is the class of all ordinals.
The {\em rank} of a set $x$ is the least ordinal $\alpha$ such that $x \in V_{\alpha+1}$ (see \cite[The paragraph after Lemma 6.3 in Section 6; Pages 64 and 65]{Jec2003}). Without $\mathsf{AC}$, a set $m$ is called a {\em cardinal} if it is the cardinality $\vert x\vert$ of some set $x$, where 
		$\vert x\vert$ = $\{y : y \sim x$ and $y$ is of least rank$\}$ where $y \sim x$ means the
		existence of a bijection $f : y \rightarrow x$ 
		(see \cite[Definition 2.2, p. 83]{Lev2002} and \cite[Section 11.2]{Jec1973}).
\end{defn}

\begin{defn}\label{Definition 2.3}
A topological space $(X,\tau)$ is called {\em compact} if for every $U \subseteq \tau$ such that $\bigcup U = X$ there is a finite subset $V\subseteq U$ such that $\bigcup V = X$.
\end{defn}

\begin{defn}\label{Definition 2.4}
   Let $\omega$ be the set of all natural numbers. A set $X$ is {\em Dedekind-finite} if it satisfies the following equivalent conditions (cf. \cite[Definition 1]{HT2006}):
	\begin{itemize}
			\item $\omega\not\preceq X$,
			\item  $A \prec X$ for every proper subset $A$ of $X$.
	\end{itemize}
    A set $X$ is {\em finite} if there exists an $n \in \omega$ and a bijection $f: X \rightarrow n$. Otherwise, $X$ is called {\em infinite}.\footnote{Clearly, finite sets are Dedekind-finite sets. In $\mathsf{ZFC}$, but not in $\mathsf{ZF}$, Dedekind-finite sets are finite; see \cite{Jec1973}.}
\end{defn}

\begin{defn}\label{Definition 2.5}
    For every family $\mathcal{B}=\{B_{i}:i\in I\}$ of non-empty sets, $\mathcal{B}$ is said to have a {\em partial choice function} if  $\mathcal{B}$ has an infinite subfamily $\mathcal{C}$ with a choice function. The following weak choice forms are equivalent to K\H{o}nig's lemma (cf. \cite{HR1998}, \cite[the proof of Theorem 4.1(i)]{DHHKR2008}):
    \begin{enumerate}
    \item $\mathsf{AC_{fin}^{\omega}}$: Any countably infinite family of non-empty finite sets has a choice function.
    \item $\mathsf{PAC^{\omega}_{fin}}$: Any countably infinite family of non-empty finite sets has a partial choice function.
    \item The union of a countably infinite family of finite sets is countably infinite.
    \end{enumerate}
\end{defn}

\section{Known Results and basic results under K\H{o}nig's Lemma}

\subsection{Known Results}

\begin{fact}\label{Fact 3.1}($\mathsf{ZF}$)
{\em The following holds:
\begin{enumerate}
    \item (Kalinowski, Pil\'{s}niak, and Wo\'{z}niak; \cite[Theorem 2.2]{KPW2016})
    If $G$ is a finite connected graph of order $n \geq 3$, then $D''(G)\leq \lceil \sqrt{\Delta(G)}\rceil$.
    \item (Caro, Petru{\v{s}}evski, and {\v{S}}krekovski; \cite[Theorem 4.9, Corollary 4.10]{CPS2022})
    If $G$ is a finite connected graph where $\Delta(G) \geq 1$, then $\chi_{o}(G) \leq 2\Delta(G)$ unless $G$ is the cycle graph $C_{5}$.
    \item (Molloy and Reed; \cite{MR1998})
    If $G$ is a finite graph, then $\chi''(G) \leq \Delta(G)+10^{26}$.
    \item (Loeb; \cite[Theorem 1]{Loeb1965})
    Let $\{X_{i}\}_ {i \in I}$ be a family of compact spaces that is indexed by a well-orderable set $I$. If there is a choice function $F$ on the collection $\{C$ : C is
    closed, $C \not= \emptyset, C \subset X_{i}$ for some $i\in I\}$, then the product space
    $\prod_{i\in I}X_{i}$ is compact in the product topology.
\end{enumerate}
}
\end{fact}

From Fact \ref{Fact 3.1}(2), we can see that for any graph $G$ (connected or disconnected) where $\Delta(G)\geq 3$, we have $\chi_o(G) \leq 2\Delta(G)$. Moreover, if $\Delta(G)\leq 2$, then $\chi_o(G) = 2\Delta(G)+1$ if and only if $G$ has $C_5$ as a connected component, otherwise $\chi_o(G)\leq 2\Delta(G)$.
\subsection{Basic results under K\H{o}nig's Lemma}

\begin{observation}\label{Observation 3.2}
{($\mathsf{ZF}$)}{\em Every graph $G=(V_{G}, E_{G})$ based on a well-ordered set of vertices has a distinguishing chromatic number, a distinguishing chromatic index, a total chromatic number, a total distinguishing number, a total distinguishing chromatic number, a neighbor-distinguishing number, a neighbor-distinguishing index, and an odd chromatic number. Moreover, if $G$ is connected, then it also has a breadth-first search (BFS) spanning tree rooted at $v_{0}$ for each $v_{0}\in V_{G}$.}
\end{observation}
\begin{proof}
If $V_{G}$ is well-orderable, then $E_{G} \subseteq [V_{G}]^{2}$ is also well-orderable. We can therefore color each vertex (or each edge) of $G$ with a unique color. We can use transfinite recursion without invoking any form of choice, and modify the $BFS$-algorithm to find a spanning tree rooted at $v_{0}$.\footnote{It is important to note that the statement ``Any connected graph contains a partial subgraph, which is a tree" is equivalent to $\mathsf{AC}$ (see H\"{o}ft--Howard \cite{HH1973}). }
\end{proof}

\begin{thm}\label{Theorem 3.3}{($\mathsf{ZF}$+ K\H{o}nig’s Lemma)}
{\em Let $G=(V_G,E_G)$ be a connected infinite graph where $\Delta(G)$ is finite. Then the
following holds:

\begin{enumerate}
    \item $D''(G)\leq \lceil \sqrt{\Delta(G)}\rceil$,
    \item $\chi_{o}(G)\leq 2\Delta(G)$,
    \item $\chi''(G)\leq \Delta(G)+10^{26}$,
    \item  $\chi''_{D}(G)\leq \Delta(G)+10^{26}+1$.
\end{enumerate}
}
\end{thm}

\begin{proof}
Choose any $r \in V_{G}$. Let $N^{0}(r)=\{r\}$. For each integer $n \geq 1$, define $N^{n}(r) = \{v \in V_{G} : d_{G}(r, v) = n\}$, where ``$d_{G}(r, v) = n$'' means that there are $n$ edges on the shortest path connecting
$r$ and $v$. Each $N^{n}(r)$ is finite since $\Delta(G)$ is finite, and $V_{G} = \bigcup_{n\in \omega}N^{n}(r)$ since $G$ is connected. 
By K\H{o}nig’s Lemma, which is equivalent to $\mathsf{AC^{\omega}_{fin}}$, $V_{G}$ is countably infinite (and thus well-orderable).

(1). Since $G$ is a connected infinite graph, $\Delta(G)\geq 2$. Let $v_0\in V_G$ such that $\deg(v_{0}) = \Delta(G)$. By Observation \ref{Observation 3.2}, there is a $BFS$ spanning tree $T$ of $G$ rooted at $v_0$ in $\mathsf{ZF}$. Assign the color $0$ to $v_0$.  
By K\H{o}nig's Lemma, there is an infinite ray $p= \langle v_i: i\in\omega\rangle$ in $T$ that starts from $v_0$.  
We will construct a total distinguishing coloring with $\lceil \sqrt{\Delta(G)}\rceil$ colors.
For a vertex $v\in V_{G}$, we denote
$$N(v)=N^{1}(v) \quad \text{ and }\quad  S(v) = \{(\{v,u\}, u) : \{v,u\} \in E_{G}\}.
$$

Pick $z\in N(v_0)\setminus \{v_1\}$ (since selecting an element from a set does not involve any form of choice) and assign the color pair $(1, 1)$  to both  $(\{v_0,v_1\}, v_1)$ and $(\{v_0,z\}, z)$. Since $\deg(v_0) = \Delta(G)$, and the pair $(1,1)$ is used twice to color the elements of $S(v_{0})$,
we can color every pair from $ S(v_0)\setminus \{(\{v_0,v_1\}, v_1), (\{v_0,z\}, z)\}$ with  distinct pairs of colors different from $(1, 1)$. Following a similar strategy as in \cite[Theorem 2.2]{KPW2016} due to Kalinowski, Pil\'{s}niak, and Wo\'{z}niak, we will 
color  $G$ such that $v_0$ will be the only vertex colored with $0$, and the pair $(1, 1)$ appears twice in $S(v_0)$.  Hence $v_0$ will be
fixed by every color preserving automorphism, and all vertices in $N(v_0)$
will be fixed, except possibly, $v_1$ and $z$. In order to distinguish $v_1$ and $z$, we color the elements from
$$\{(\{v_1,u\}, u)\in S(v_1):  u \in N^2(v_0)\}\quad \text{ and }\quad \{(\{z,u\}, u) \in  S(z) : \{z,u\} \in E_{T}, u \in N^2(v_0)\}$$
with two distinct sets of color pairs. This is possible since each of these sets contains
at most $\Delta(G)-1$ pairs, while  at least $\Delta(G)$-many distinct color pairs are available. Now suppose that for some $k\in \omega$, every element  from $u\in \bigcup_{i\leq k} N^i(v_0)$ are colored such that they are fixed by every color preserving automorphism. For every $u\in N^k(v_0)$, color $(\{u,z\}, z)$, where $z$ is a child of $u$ in $T$, with distinct color pairs except  $(1, 1)$. 
Hence, all neighbors in $N(u)\cap N^{k+1}(v_0)$ will be fixed. Therefore, the elements of $\bigcup_{i\leq k+1}N^i(v_0)$ are also fixed by every color preserving automorphism. 

Continuing in this fashion by mathematical induction (and noting that the process will not stop at a finite stage), we end up coloring the vertices and edges of $T$. Finally, we color the edges in  $E_{G} \setminus E_{T}$ with $0$. The color assignment of $G$, defined above, is total distinguishing.

(2).  
By Fact \ref{Fact 3.1}(2), every finite subgraph of $G$ is $(2\Delta(G) + 1)$-odd colorable. 
Endow the set $\{1,...,2\Delta(G)+1\}$ with the discrete topology and $T=\{1,...,2\Delta(G)+1\}^{V_{G}}$ with the product topology. Since $V_{G}$ is well-orderable, $\{1,...,2\Delta(G)+1\}\times {V_{G}}$ is well-orderable. By Fact \ref{Fact 3.1}(4), $T$ is compact.
We can use Fact \ref{Fact 3.1}(2) and the compactness of $T$ to prove that $G$ is $(2\Delta(G) + 1)$-odd colorable similarly to the proof of the de Bruijn–Erd\H{o}s theorem, without invoking any form of choice.
If $G$ were not $2\Delta(G)$-odd colorable, then there would exist a finite subgraph $H$ of $G$ with $\chi_{o}(H)= 2\Delta(G) + 1\geq 2\Delta(H) + 1$. 

By the paragraph after Fact \ref{Fact 3.1}, we have $\Delta(H)\leq 2$ and $\chi_{o}(H)= 2\Delta(H) + 1$. 
Moreover, $H$ would have $C_5$ as a connected component, and so $\Delta(H)\neq 1$. Thus, $\Delta(H)=2$ and $\chi_{0}(H)=2\Delta(G) + 1=2\Delta(H) + 1=5$. Consequently, $\Delta(G)=2$. Since $G$ is connected and $C_{5}$ is a connected component of $H$, there exists a vertex $x$ of $C_{5}$ such that the degree of $x$ in $G$ is greater than or equal to 3.
Thus, $G$ would have to contain a vertex of degree greater
than $\Delta(G)$, which is impossible.

(3). This follows from Fact \ref{Fact 3.1}(3), and a straightforward compactness argument as in (2). 

(4). This follows from (3) and the fact that the statement ``$\chi''_{D}(G)$ is at most $\chi''(G)+1$'' holds in $\mathsf{ZF}$ since $V_{G}$ is well-orderable (see \cite[proof of Theorem 13]{IKPS2017}). 
\end{proof}
\section{Total, Neighbor-distinguishing, and distinguishing proper colorings}

In this section, we prove the following.

\begin{thm}\label{Theorem 4.1}{($\mathsf{ZF}$)}{\em The following statements are equivalent:
\begin{enumerate}
\item K\H{o}nig’s Lemma.
\item Every infinite locally finite connected graph has a distinguishing chromatic number.

\item Every infinite locally finite connected graph has a distinguishing chromatic index.

\item Every infinite locally finite connected graph has a total chromatic number.

\item Every infinite locally finite connected graph has a total distinguishing chromatic number.

\item Every infinite locally finite connected graph has a neighbor-distinguishing index.
\end{enumerate}
}
\end{thm}

Let $G=(V_{G}, E_{G})$ be an infinite, locally finite, connected graph. Following the arguments in the first paragraph of the proof of Theorem \ref{Theorem 3.3}, $V_{G}$ is well-orderable by K\H{o}nig’s Lemma.
Thus, (1) implies (2)-(6) by Observation \ref{Observation 3.2}. 
We show that each of the conditions
(2)-(6) implies (1).
Since $\mathsf{AC_{fin}^{\omega}}$ is equivalent to its partial version $\mathsf{PAC_{fin}^{\omega}}$  (cf. Definition \ref{Definition 2.5}), it suffices to show that each of (2)-(6) implies $\mathsf{PAC_{fin}^{\omega}}$.
Let $\mathcal{A}=\{A_{n}:n\in \omega\}$ be a countably infinite set of non-empty finite sets without a partial choice function. Without loss of generality, we assume that $\mathcal{A}$ is disjoint. 
It is enough to show that there exists an infinite locally finite connected graph $G$ such that $G$ does not have the coloring parameters mentioned in (2)-(6).

\subsection{Construction of the graph}

Pick a countably infinite sequence $T=\{t_{n}:n\in \omega\}$ disjoint from $A=\bigcup_{i\in\omega}A_{i}$. Pick $t', t''\not\in (\bigcup_{n\in\omega}A_n)\cup T$ such that $t''\neq t'$ and consider the following infinite locally finite connected graph $G=(V_{G}, E_{G})$ (see Figure \ref{Figure 1}):

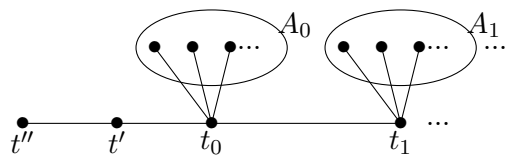
\begin{figure}[!ht]
\centering
\begin{minipage}{\textwidth}
\centering
\begin{tikzpicture}[scale=0.5]
\draw (-2.5, 1) ellipse (2 and 1);
\draw (-4,1) node {$\bullet$};
\draw (-3,1) node {$\bullet$};
\draw (-2,1) node {$\bullet$};
\draw (-1.5,1) node {...};
\draw (-0.3,1.6) node {$A_{0}$};
\draw (-2.5,-1) node {$\bullet$};
\draw (-2.5,-1.5) node {$t_{0}$};
\draw (-4,1) -- (-2.5,-1);
\draw (-3,1) -- (-2.5,-1);
\draw (-2,1) -- (-2.5,-1);
\draw (2.5, 1) ellipse (2 and 1);
\draw (1,1) node {$\bullet$};
\draw (2,1) node {$\bullet$};
\draw (3,1) node {$\bullet$};
\draw (3.5,1) node {...};
\draw (4.7,1.6) node {$A_{1}$};
\draw (2.5,-1) node {$\bullet$};
\draw (2.5,-1.5) node {$t_{1}$};
\draw (3,1) -- (2.5,-1);
\draw (2,1) -- (2.5,-1);
\draw (1,1) -- (2.5,-1);
\draw (-2.5,-1) -- (2.5,-1);
\draw (3.5,-1) node {...};
\draw (5,1) node {...};
\draw (-5,-1) node {$\bullet$};
\draw (-5,-1.5) node {$t'$};
\draw (-5,-1) -- (-2.5,-1);
\draw (-7.5,-1) node {$\bullet$};
\draw (-7.5,-1.5) node {$t''$};
\draw (-7.5,-1) -- (-5,-1);
\end{tikzpicture}
\end{minipage}
\caption{\em Graph $G$, an infinite locally finite connected graph.}
\label{Figure 1}
\end{figure}

\begin{alignat*}{3}
V_{G} &:= (\bigcup_{n\in\omega}A_n)\cup T \cup \{t',t''\},     \\
E_{G} &:= \{t'', t'\} \cup \{t', t_0\} \cup \big\{ \{t_n,t_{n+1}\}: n\in\omega\big\} &&\cup\big\{ \{t_n,x\}: n\in\omega,x\in A_n \big\}. \\
\end{alignat*}

\subsection{{\em G} has no distinguishing chromatic number}
In this subsection, we show that condition (2) implies $\mathsf{PAC_{fin}^{\omega}}$.
The following two claims were proved in \cite{BMG2024}.

\begin{claim}{(\cite[Claim 4.7]{BMG2024})}\label{claim 4.2}
{\em 
 Any automorphism of $G$ fixes the vertices $t', t''$, and $t_m$ for each non-negative integer $m$.
}
\end{claim}

\begin{claim}{(\cite[Claim 4.8]{BMG2024})}\label{claim 4.3}
{\em Fix $p\in \omega$ and $x\in A_p$. Then $Orb_{Aut(G)}(x)=\{g(x): g\in Aut(G)\}=A_p$.} 
\end{claim}


\begin{claim}\label{claim 4.4}
{\em If $f: V_{G} \rightarrow C$ is a distinguishing proper vertex coloring of $G$, then the following holds:
\begin{enumerate}
    \item If $x,y\in A_{i}$ such that $x\neq y$, then $f(x)\neq f(y)$.
    \item For each $c \in C$, the set $M_{c}=f^{-1}(c) \cap \bigcup_{i\in\omega} A_{i}$ must be finite.
    \item $f[\bigcup_{n\in \omega}A_{n}]$ is infinite.
    \item $f[\bigcup_{n\in \omega}A_{n}]$ is Dedekind-finite.
    \item For each natural number $m$, if  $c_0, ...,c_m\in f[\bigcup_{n\in \omega}A_{n}]$, then there is a proper vertex coloring $h: \bigcup_{n\in \omega}A_{n} \to f[\bigcup_{n\in \omega}A_{n}]\setminus\{c_0, ..., c_m\}$ such that $h(x) = f(x)$ whenever $f(x)\neq c_i$.\footnote{The crux of this claim is inspired by the arguments of Herrlich–-Tachtsis \cite[Proposition 23]{HT2006} where they proved that no Russell graph has a chromatic number in $\mathsf{ZF}$.}
\end{enumerate}
}
\end{claim}
\begin{proof}
(1). This follows by Claim \ref{claim 4.3}. 

(2). If $M_{c}$ is infinite, then it will generate a partial choice function for $\mathcal{A}$ by (1).

(3). If $f[\bigcup_{n\in \omega}A_{n}]$ is finite, then $\bigcup_{n\in \omega}A_{n}=\bigcup_{c\in f[\bigcup_{n\in \omega}A_{n}]} M_{c}$
is finite by (2), since the finite union of finite sets is finite in $\mathsf{ZF}$. Thus, we obtain a contradiction. 

(4). We note that $\bigcup_{n\in \omega}A_{n}$ is Dedekind-finite since $\mathcal{A}$ has no partial choice function. For the sake of contradiction, assume that $C=\{c_{i}:i\in \omega\}$ is a countably infinite subset of $f[\bigcup_{n\in \omega}A_{n}]$. 
Fix a well-ordering of $\mathcal{A}$ (since $\mathcal{A}$ is countable, and hence well-orderable).
Define $d_{i}$ to be the {\em unique} element of $M_{c_{i}} \cap A_{n}$ 
where $n$ is the least element of $\{m\in\omega: M_{c_{i}} \cap A_{m}\neq \emptyset\}$ with respect to the well-ordering of $\mathcal{A}$.
Such an $n$ exists since $c_{i} \in f[\bigcup_{n<\omega} A_{n}]$ and
$M_{c_{i}} \cap A_{n}$ has a single element by the property mentioned in (1).
Then $\{d_{i}:i\in\omega\}$ is a countably infinite subset of $\bigcup_{n\in \omega}A_{n}$ which contradicts the fact that $\bigcup_{n\in \omega}A_{n}$ is Dedekind-finite.

(5). Fix some $c_{0} \in f[\bigcup_{n\in \omega}A_{n}]$.
Then $Index(M_{c_{0}}) = \{n \in \omega : M_{c_{0}} \cap A_{n} \neq \emptyset\}$ is finite. By (3), there exists some
$b_{0} \in (f[\bigcup_{n\in \omega}A_{n}]\backslash \bigcup_{m\in Index(M_{c_{0}})} f[A_{m}])$ (see Figure \ref{Figure 2}).

\begin{figure}[!ht]
\centering
\begin{minipage}{\textwidth}
\centering
\begin{tikzpicture}[scale=0.7]
\draw (-2.5, 1) ellipse (1 and 1.5);
\draw (-2,0) node {$\bullet$};
\draw (-2.5,0.5) node {$\bullet$};
\draw (-3,1) node {$\bullet$};
\draw (-2.5,1.5) node {$...$};
\draw (-1,2) node {$A_{0}$};
\draw (-2.5,-1) node {$\bullet$};
\draw (-2.5,-1.5) node {$t_{0}$};
\draw (-2,0) -- (-2.5,-1);
\draw (-2.5,0.5) -- (-2.5,-1);
\draw (-3,1) -- (-2.5,-1);

\draw (2.5, 1) ellipse (1 and 1.5);
\draw (3,0) node {$\bullet$};
\draw (2.5,0.5) node {$\bullet$};
\draw (2,1) node {$\bullet$};
\draw (2.5,1.5) node {...};
\draw (4,2) node {$A_{1}$};
\draw (2.5,-1) node {$\bullet$};
\draw (2.5,-1.5) node {$t_{1}$};
\draw (3,0) -- (2.5,-1);
\draw (2.5,0.5) -- (2.5,-1);
\draw (2,1) -- (2.5,-1);

\draw (-2.5,-1) -- (2.5,-1);
\draw (3.5,-1) node {...};
\draw (5,1) node {...};

\draw (6.5, 1) ellipse (1 and 1.5);
\draw (7,0) node {$\bullet$};
\draw (6.5,0.5) node {$\bullet$};
\draw (6,1) node {$\bullet$};
\draw (2.5,1.5) node {...};
\draw (8,2) node {$A_{i}$};
\draw (6.5,-1) node {$\bullet$};
\draw (6.5,-1.5) node {$t_{i}$};
\draw (7,0) -- (6.5,-1);
\draw (6.5,0.5) -- (6.5,-1);
\draw (6,1) -- (6.5,-1);

\draw (8.5,-1) node {...};
\draw (10,1) node {...};


\draw (-10,-4) rectangle (10,-1.8);
\draw (10,-1.4) node {$C$};

\draw [dashed](-9,-3.8) rectangle (9,-2);
\draw (7.2,-2.6) node {$f[\bigcup_{n\in \omega}A_{n}]$};

\draw (-7.5,-2.6) node {$\bullet c_{0}$};
\draw (-8.5,-3.6) rectangle (2,-2.2);
\draw (-1.2,-2.8) node {$\bigcup_{m\in Index(M_{c_{0}})} f[A_{m}]$};
\draw (3,-2.6) node {$\bullet b_{0}$};

\draw [dashed](-2.2,-0.2) rectangle (8,0.3);
\draw (8.7,0.1) node {$M_{c_{0}}$};

\draw[-triangle 60] (-4,0.5) [bend right=60] to (-5,-2.1);
\draw (-5.7,-0.5) node {$f$};
\draw (-4,-0.7) rectangle (9.5,3);
\draw (-6.5,2.5) node {$\bigcup_{m\in Index(M_{c_{0}})} A_{m}$};
\draw (-3.5,-1) node {$\bullet$};
\draw (-3.5,-1.5) node {$t'$};
\draw (-3.5,-1) -- (-2.5,-1);
\draw (-4.5,-1) node {$\bullet$};
\draw (-4.5,-1.5) node {$t''$};
\draw (-4.5,-1) -- (-3.5,-1);
\end{tikzpicture}
\end{minipage}
\caption{\em There exists a color
$b_{0}$ from the set $(f[\bigcup_{n\in \omega}A_{n}]\backslash \bigcup_{m\in Index(M_{c_{0}})} f[A_{m}])$.}
\label{Figure 2}
\end{figure}
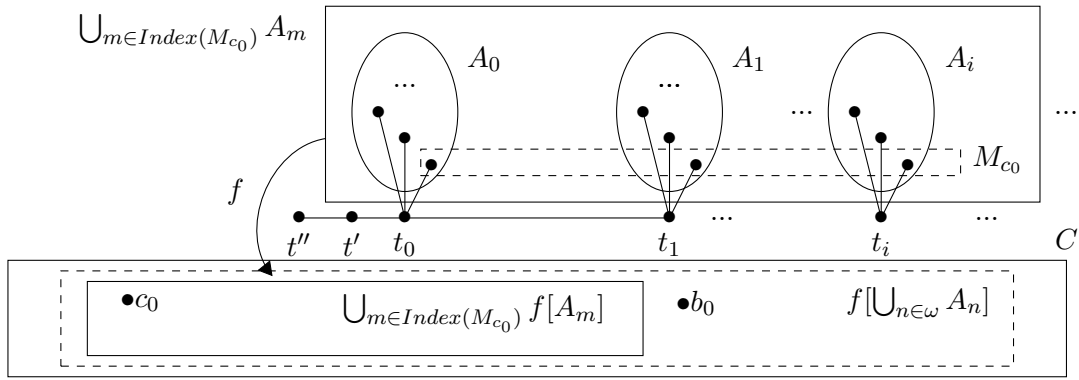

Define a vertex coloring $g: \bigcup_{n\in \omega}A_{n} \rightarrow (f[\bigcup_{n\in \omega}A_{n}]\backslash \{c_{0}\})$ as follows:

\begin{center}
$g(x) =
\begin{cases} 
f(x) & \text{if}\, f(x)\neq c_{0}, \\

b_{0} & \text{otherwise}.
\end{cases}$
\end{center}

Similar to the definition of $g$, if $c_{0},...,c_{m}\in f[\bigcup_{n\in \omega}A_{n}]$, then one can define a vertex coloring 
$h: \bigcup_{n\in \omega}A_{n}\rightarrow (f[\bigcup_{n\in \omega}A_{n}]\backslash \{ c_{0},...,c_{m}\})$
satisfying the properties of the assertion.
\end{proof}
\begin{claim}\label{claim 4.5}
{\em $G$ has no distinguishing chromatic number.}
\end{claim}
\begin{proof}
Assume that $G$ has a distinguishing chromatic number (say $\kappa$). Let  $f: V_{G} \to C$ be a distinguishing proper vertex coloring with $|C|=\kappa$. 
For some $c_{0}, c_{1},c_{2}\in f[\bigcup_{n\in \omega}A_{n}]$, we can define a vertex coloring $h: \bigcup_{n\in \omega}A_{n}\rightarrow (f[\bigcup_{n\in \omega}A_{n}]\backslash \{ c_{0},c_{1},c_{2}\})$ as in Claim \ref{claim 4.4}(5).
Let $h(t'')=c_{0}$, $h(t')=c_{1}$, $h(t_{2n})=c_{0}$, and $h(t_{2n+1})=c_{1}$ for all $n\in\omega$. 
By Claim \ref{claim 4.2}, $h:V_{G} \rightarrow (f[\bigcup_{n\in \omega}A_{n}]\backslash \{c_{2}\})$ is a distinguishing proper vertex coloring of $G$. 
Define $C_{1}=f[\bigcup_{n\in\omega} A_{n}]\backslash \{c_{2}\}$. By Claim \ref{claim 4.4}(4),
$C_{1} \prec f[\bigcup_{n\in\omega} A_{n}] \preceq C$ which contradicts the
fact that $\chi_{D}(G)=\kappa$. 
\end{proof}

Thus, condition (2) implies $\mathsf{PAC_{fin}^{\omega}}$. This shows that condition (2) in Theorem \ref{Theorem 4.1} is equivalent to K\H{o}nig’s Lemma.
\subsection{A Lemma on edge colorings} By Claim \ref{claim 4.2}, every automorphism fixes the edges $\{t'',t'\}$, $\{t', t_0\}$, and $\{t_n, t_{n+1}\}$ for each $n\in \omega$. The following lemma will be applied in subsections 4.4, 4.5, 4.6, and 4.7. Throughout the section, let 
$B = \{\{t_n,x\}: n\in \omega,  x\in A_n\}$.

\begin{lem}\label{lemma 4.6}
{\em If $f$ is a distinguishing proper edge coloring (resp. neighbor-distinguishing edge coloring, total coloring) of $G$, then the following holds:

\begin{enumerate}
    \item For each $n\in \omega$ and $x,y\in A_n$ such that $x\neq y$, $f(\{t_n, x\})\neq f(\{t_n, y\})$.
    \item $f[B]$ is an infinite, Dedekind-finite set. 
    \item For each natural number $m$, if  $c_0, ...,c_m\in f[B]$, then there is an edge coloring $h: B \to f[B]\setminus\{c_0, ..., c_m\}$ such that $h(x) = f(x)$ whenever $f(x)\neq c_i$. 
\end{enumerate}
} 
\end{lem}
\begin{proof}
    This follows modifying the arguments of Claims \ref{claim 4.3}, and \ref{claim 4.4}.
\end{proof}

\subsection{{\em G} has no distinguishing chromatic index}
In this subsection, we show that condition (3) implies $\mathsf{PAC_{fin}^{\omega}}$.
Assume that the graph $G$ has a distinguishing chromatic index (say $\kappa$), and let $f: E_{G}\to C$ be a distinguishing proper edge coloring with $\vert C\vert=\kappa$.
By Lemma \ref{lemma 4.6}, $f[B]$ is an infinite, Dedekind-finite set. Moreover, for some $c_0, c_1,c_2\in f[B]$, we can define an edge coloring $h: B \to f[B]\setminus\{c_0, c_1, c_2\}$ as in Lemma \ref{lemma 4.6}(3). 
Define,
\begin{enumerate}
    \item $h(\{t'', t'\})= c_0$, $h(\{t', t_{0}\})= c_1$,
    \item $h(\{t_{2n}, t_{2n+1}\})= c_0$ and $h(\{t_{2n+1}, t_{2n+2}\})= c_1$ for all $n\in \omega$. 
\end{enumerate}
We obtain a distinguishing proper edge coloring $h:E_{G}\to f[B]\setminus\{c_2\}$, with 
$f[B]\setminus\{c_2\} \prec f[B]\preceq C$ 
as $f[B]$ is Dedekind-finite, contradicting the fact that $\chi'_{D}(G)=\kappa$.

This shows that K\H{o}nig’s Lemma is equivalent to condition (3) in Theorem \ref{Theorem 4.1}.

\subsection{{\em G} has no neighbor-distinguishing index}
Here we show that condition (6) implies $\mathsf{PAC_{fin}^{\omega}}$.
Assume that $G$ has a neighbor-distinguishing index (say $\kappa$), and let $f: E_{G}\to C$ be a neighbor-distinguishing edge coloring with $|C|=\kappa$. 
By Lemma \ref{lemma 4.6}, $f[B]$ is an infinite, Dedekind-finite set and for some $c_0, c_1,c_2,c_3\in f[B]$, we can define an edge coloring $h: B \to f[B]\setminus\{c_0, c_1, c_2, c_3\}$ as in Lemma \ref{lemma 4.6}(3). 
Define, 
\begin{enumerate}
    \item $h(\{t'', t'\})= c_1$, $h(\{t', t_{0}\})= c_2$,
    \item $h(\{t_{3n}, t_{3n+1}\})= c_0$, $h(\{t_{3n+1}, t_{3n+2}\})= c_1$, and $h(\{t_{3n+2}, t_{3n+3}\})= c_2$ for all $n\in \omega$. 
\end{enumerate}

Then, $h:E_{G}\to f[B]\setminus\{c_3\}$ is a neighbor-distinguishing edge coloring of $G$. Finally, 
$f[B]\setminus\{c_{3}\} \prec  f[B] \preceq C$
contradicts the fact that $\kappa$ is the neighbor-distinguishing index of $G$.    

Thus, K\H{o}nig’s Lemma is equivalent to condition (6) in Theorem \ref{Theorem 4.1}.

\subsection{{\em G} has no total chromatic number}
We show that condition (4) implies $\mathsf{PAC_{fin}^{\omega}}$.
Assume that $G$ has a total chromatic number. Let  $f: V_{G}\cup E_{G} \to C$ be a total coloring with $|C|=\kappa$, where $\kappa$ is the total chromatic number of $G$. 
By Lemma \ref{lemma 4.6}, $f[B]$ is an infinite, Dedekind-finite set and for some $c_i\in f[B]$ where $0\leq i\leq 5$, we can define an edge coloring $h: B \to f[B]\setminus\{c_0, c_1, c_2, c_3, c_4, c_5\}$ as in Lemma \ref{lemma 4.6}(3). 
Let,
\begin{enumerate}
    \item $h(\{t'', t'\})= c_0$ and $h(\{t', t_{0}\})= c_1$.
    \item $h(\{t_{2n}, t_{2n+1}\})= c_0$ and $h(\{t_{2n+1}, t_{2n+2}\})= c_1$ for all $n\in \omega$. 
    \item $h(t'')=c_{2}$ and $h(t')=c_{3}$.
    \item $h(t_{2n})=c_{2}$ and $h(t_{2n+1})=c_{3}$ for all $n\in\omega$.
    \item $h(x)=c_{4}$ for all $x\in \bigcup_{i\in \omega} A_{i}$. 
\end{enumerate}

Then, $h: V_{G}\cup E_{G} \rightarrow f[B]\backslash \{c_{5}\}$ is a total coloring of $G$ and 
$f[B]\setminus\{c_{5}\} \prec  f[B] \preceq C$
contradicts the fact that $\kappa$ is the total chromatic number of $G$.    

This shows that K\H{o}nig’s Lemma is equivalent to condition (4) in Theorem \ref{Theorem 4.1}.

\subsection{{\em G} has no total distinguishing chromatic number} 
Here we show that condition (5) implies $\mathsf{PAC_{fin}^{\omega}}$.
Assume that $G$ has a total distinguishing chromatic number, and let  $f: V_{G}\cup E_{G} \to C$ be a properly total coloring with $|C|=\kappa$ such that only the trivial automorphism
preserves it, where $\kappa$ is the total distinguishing chromatic number of $G$. 
We claim that the coloring $h$ in subsection 4.6 is defined in a way so that only the trivial automorphism 
preserves $h$. 
In particular, by Claim \ref{claim 4.2}, the vertices $t', t''$, and $t_m$ are fixed by every automorphism for each non-negative integer $m$. Since $h$ is a properly total coloring, any neighbor of $t_{m}$ has to be mapped into itself because every edge incident
to $t_{m}$ has a distinct color by Lemma \ref{lemma 4.6}(1).
Thus, following the arguments in subsection 4.6, we obtain a contradiction. So K\H{o}nig’s Lemma is equivalent to condition (5) in Theorem \ref{Theorem 4.1}.

This completes the proof of Theorem 4.1.

\section{Neighborhood distinguishing and Odd vertex colorings}

We prove the following.
\begin{thm}\label{Theorem 5.1}{($\mathsf{ZF}$)}{\em The following statements are equivalent:
\begin{enumerate}
\item K\H{o}nig’s Lemma.
\item Every infinite locally finite connected graph has an odd chromatic number.
\item Every infinite locally finite connected graph has a neighbor-distinguishing number.
\end{enumerate}
}
\end{thm}

Implications (1)$\Rightarrow$(2)-(3) follow from Observation \ref{Observation 3.2}, and the fact that $\mathsf{AC_{fin}^{\omega}}$ implies ``Every infinite locally finite connected graph is countably infinite''.
In view of the proof of Theorem \ref{Theorem 4.1}, it suffices to show that 
each of the conditions (2)-(3) implies $\mathsf{PAC_{fin}^{\omega}}$. Assume $\mathcal{A}$ and $T$ as in the proof of Theorem \ref{Theorem 4.1}.
Consider the following infinite locally finite connected graph $G_{1}=(V_{G_{1}}, E_{G_{1}})$ (see Figure \ref{Figure 3}):
  
		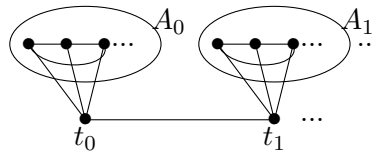
\begin{figure}[!ht]
			\centering
			\begin{minipage}{\textwidth}
				\centering
				\begin{tikzpicture}[scale=0.5]
					\draw (-2.5, 1) ellipse (2 and 1);
					\draw (-4,1) node {$\bullet$};
					\draw (-3,1) node {$\bullet$};
					\draw (-2,1) node {$\bullet$};
					\draw (-1.5,1) node {...};
					\draw (-0.3,1.6) node {$A_{0}$};
					\draw (-4,1) -- (-3,1);
					\draw (-3,1) -- (-2,1);
					\draw (-2,1) to[out=-70,in=-70] (-4,1);
					\draw (-2.5,-1) node {$\bullet$};
					\draw (-2.5,-1.5) node {$t_{0}$};
					\draw (-4,1) -- (-2.5,-1);
					\draw (-3,1) -- (-2.5,-1);
					\draw (-2,1) -- (-2.5,-1);
					\draw (2.5, 1) ellipse (2 and 1);
					\draw (1,1) node {$\bullet$};
					\draw (2,1) node {$\bullet$};
					\draw (3,1) node {$\bullet$};
					\draw (3.5,1) node {...};
					\draw (4.7,1.6) node {$A_{1}$};
					\draw (1,1) -- (2,1);
					\draw (2,1) -- (3,1);
					\draw (1,1) to[out=-70,in=-70] (3,1);
					\draw (2.5,-1) node {$\bullet$};
					\draw (2.5,-1.5) node {$t_{1}$};
					\draw (3,1) -- (2.5,-1);
					\draw (2,1) -- (2.5,-1);
					\draw (1,1) -- (2.5,-1);
					\draw (-2.5,-1) -- (2.5,-1);
					\draw (3.5,-1) node {...};
					\draw (5,1) node {...};
				\end{tikzpicture}
			\end{minipage}
			\caption{\em Graph $G_{1}$, an infinite locally finite connected graph.}
            \label{Figure 3}
		\end{figure}
  
\begin{alignat*}{3}
V_{G_1} &:= (\bigcup_{n\in\omega}A_n)\cup T,     \\
E_{G_1} &:= \big\{ \{t_n,t_{n+1}\}: n\in\omega\big\} &&\cup\big\{ \{t_n,x\}: n\in\omega,x\in A_n \big \} 
&&\cup\big\{ \{x,y\}: n\in\omega,x,y\in A_n,x\neq y \big\}.  
\end{alignat*}

\begin{claim}\label{claim 5.2}
{\em If $f: V_{G_{1}}\rightarrow C$ is an odd proper vertex coloring (resp. neighbor-distinguishing vertex coloring) of $G_{1}$, then the following holds:

\begin{enumerate}
    \item $f[\bigcup_{n\in \omega}A_{n}]$ is an infinite, Dedekind-finite set. 
    \item For each natural number $m$, if  $c_0, ...,c_m\in f[\bigcup_{n\in \omega}A_{n}]$, then there is 
    an odd proper vertex coloring (resp. neighbor-distinguishing vertex coloring)
    $h: \bigcup_{n\in \omega}A_{n} \to f[\bigcup_{n\in \omega}A_{n}]\setminus\{c_0, ..., c_m\}$ such that $h(x) = f(x)$ whenever $f(x)\neq c_i$.
\end{enumerate}
} 
\end{claim}
\begin{proof}
    This follows modifying the arguments of Claims \ref{claim 4.3}, and \ref{claim 4.4} (cf. \cite[the proof of Theorem 4.2]{BMG2024} as well as \cite[the proof of Proposition 23]{HT2006}).
\end{proof}

\subsection{$G_{1}$ has no odd chromatic number}
Assume that $G_{1}$ has an odd chromatic number (say $\kappa$), and let  $f: V_{G_{1}} \to C$ be a $C$-odd proper vertex coloring with $|C|=\kappa$. 
By Claim \ref{claim 5.2}(1), $f[\bigcup_{n\in \omega}A_{n}]$ is an infinite, Dedekind-finite set. Consider the vertex coloring $h: \bigcup_{n\in \omega}A_{n} \to f[\bigcup_{n\in \omega}A_{n}]\setminus\{c_0, c_1, c_2\}$ for some $c_0, c_1,c_2\in f[\bigcup_{n\in \omega}A_{n}]$ as in Claim \ref{claim 5.2}(2). 
Let,
    $h(t_{2n})=c_{0}$, and $h(t_{2n+1})=c_{1}$ for all $n\in\omega$. 
Then, $h:V_{G_{1}} \rightarrow (f[\bigcup_{n\in \omega}A_{n}]\backslash \{c_{2}\})$ is a proper vertex coloring of $G_{1}$. 
The following arguments show that $h$ is an odd proper vertex coloring:
\begin{enumerate}
    \item If $t_{i}\in T$, then pick the color $c$
of an arbitrary vertex of $A_{i}$, say $a_{i}$. Since $h$ is a proper vertex coloring, $h^{-1}(c)\cap N(t_{i})=\{a_{i}\}$ where $N(v)=\{u \in V_{G_{1}} : \{u,v\} \in E_{G_{1}}\}$ is the open neighborhood of $v$ for any $v\in V_{G_{1}}$. 

   \item If $x\in A_{m}$ for some $m\in\omega$, then let $c\in \{c_{0},c_{1}\}$ be the color of $t_{m}$. Clearly, $h^{-1}(c)\cap N(x)=\{t_{m}\}$ as neither $c_{0}$ nor $c_{1}$ belongs to $h[\bigcup_{n\in \omega}A_{n}]$ because $h$ is a proper vertex coloring.
\end{enumerate}

We define $C_{1}=f[\bigcup_{n\in\omega} A_{n}]\backslash \{c_{2}\}$. 
Since $f[\bigcup_{n\in \omega}A_{n}]$ is Dedekind-finite,
$C_{1} \prec f[\bigcup_{n\in\omega} A_{n}] \preceq C$. This contradicts the
fact that $\kappa$ is the odd chromatic number of $G_{1}$. 

\subsection{$G_{1}$ has no neighbor-distinguishing number}
Assume that $G_{1}$ has a neighbor-distinguishing number (say $\kappa$), and let $f: V_{G_{1}} \to C$ be a neighbor-distinguishing vertex coloring with $|C|=\kappa$. For some $c_i\in f[\bigcup_{n\in \omega}A_{n}]$ where $0\leq i\leq 3$, we pick a vertex coloring $h: \bigcup_{n\in \omega}A_{n}\rightarrow (f[\bigcup_{n\in \omega}A_{n}]\backslash \{ c_{0},c_{1},c_{2}, c_{3}\})$ as in Claim \ref{claim 5.2}(2). Let,
$h(t_{3n})=c_{0}$, $h(t_{3n+1})=c_{1}$, and $h(t_{3n+2})=c_{2}$, for all $n\in\omega$. 
Then, $h: V_{G_{1}} \rightarrow f[\bigcup_{n\in\omega} A_{n}]\backslash \{c_{3}\}$ is a neighbor-distinguishing vertex coloring of $G_{1}$. The rest follows by the arguments of subsection 5.1.

\section{Analogs of Brooks' theorem fail in ZF}
We apply the methods of sections 4 and 5 to show that the results in Fact \ref{Fact 1.1} and Theorem \ref{Theorem 3.3} fail in a model of $\mathsf{ZF}$. 

\begin{thm}\label{Theorem 6.1}
{\em 
There is a model of $\mathsf{ZF}$ where the following statements fail:
\begin{enumerate}
    \item If $G$ is a connected infinite graph where $\Delta(G)$ is finite, then $\chi(G)\leq \Delta(G)$,
    \item If $G$ is a connected infinite graph where $\Delta(G)$ is finite, then $\chi'(G)\leq \Delta(G)+1$,
    \item If $G$ is a connected infinite graph where $\Delta(G)$ is finite, then $D(G)\leq \Delta(G)$,
    \item If $G$ is a connected infinite graph where $\Delta(G)\geq 3$ is finite, then $D(G)\leq \Delta(G)-1$,
    \item If $G$ is a connected infinite graph where $\Delta(G)\geq 3$ is finite, then $D'(G)\leq \Delta(G)-1$,
    \item If $G$ is a connected infinite graph where $\Delta(G)$ is finite, then $\chi_{D}(G)\leq 2\Delta(G)-1$,
    \item If $G$ is a connected infinite graph where $\Delta(G)$ is finite, then $\chi'_{D}(G)\leq \Delta(G)+2$,
    \item If $G$ is a connected infinite graph where $\Delta(G)$ is finite, then $\chi''(G)\leq \Delta(G)+10^{26}$,
    \item If $G$ is a connected infinite graph where $\Delta(G)$ is finite, then $\chi''_{D}(G)\leq \Delta(G)+10^{26}+1$,
    
    \item If $G$ is a connected infinite graph where $\Delta(G)$ is finite, then $\chi_{o}(G)\leq 2\Delta(G)$,

    \item If $G$ is a connected infinite graph where $\Delta(G)$ is finite, then $D''(G)\leq \lceil \sqrt{\Delta(G)}\rceil$.
\end{enumerate}
Thus, the results in Fact \ref{Fact 1.1}(1)-(8) fail in the model.
}
\end{thm}

\begin{proof}
Fix any $n\in \omega\backslash \{0,1\}$. The authors of \cite{DHHKR2008} constructed a model of $\mathsf{ZF}$ where the statement ``Any countably infinite family of $n$-element sets has a partial choice function", i.e. $\mathsf{Form}$ 373($n$) in \cite{HR1998}, fails (see the paragraph after \cite[Corollary 3.2]{DHHKR2008}). 
In particular, $\mathsf{Form}$ 373($n$) fails in the model $\mathcal{N}_{2}(n)$ of \cite{HR1998} (a model of $\mathsf{ZFA}$ i.e., Zermelo-Fraenkel set theory with the Axiom of Extensionality weakened to allow the existence of atoms) (see \cite{HR1998}). 
Since $\phi=``\neg\mathsf{Form}$ 373($n$)'' is a boundable statement, there is a model $\mathcal{M}(n)$ of $\mathsf{ZF}$ where $\phi$ holds.\footnote{We refer the reader to \cite[Note 103, pages 283–286]{HR1998} for the definition of the term ``boundable statement'' and the fact that a boundable sentence is transferable to $\mathsf{ZF}$.}

Let $\mathcal{A}=\{A_{i}:i\in \omega\}$ be a countably infinite family of $n$-element sets without a partial choice function in $\mathcal{M}(n)$. Consider the graphs $G$ and $G_{1}$ from the proof of Theorems \ref{Theorem 4.1} and \ref{Theorem 5.1} (see Figures \ref {Figure 1} and \ref{Figure 3}). Clearly,  $\Delta(G)=\Delta(G_{1})=n+2\geq 4$. By the arguments in the proof of Theorems \ref{Theorem 4.1} and \ref{Theorem 5.1}, $\chi_{D}(G)$, $\chi'_{D}(G)$, $\chi''(G)$, $\chi''_{D}(G)$, and $\chi_{o}(G_{1})$ do not exist. 
Let $H_{1}$ be the graph obtained from the graph $G_{1}$ of Figure \ref{Figure 3} after deleting the edge set $\{\{x,y\}: m\in \omega, x,y\in A_m, x\neq y\}$ (see Figure \ref{Figure 4}). Clearly, $H_1$ is a connected infinite graph where $\Delta(H_{1})=n+2$.

\begin{figure}[!ht]
\centering
\begin{minipage}{\textwidth}
\centering
\begin{tikzpicture}[scale=0.5]
\draw (-2.5, 1) ellipse (2 and 1);
\draw (-4,1) node {$\bullet$};
\draw (-3,1) node {$\bullet$};
\draw (-2,1) node {$\bullet$};
\draw (-1.5,1) node {...};
\draw (-0.3,1.6) node {$A_{0}$};
\draw (-2.5,-1) node {$\bullet$};
\draw (-2.5,-1.5) node {$t_{0}$};
\draw (-4,1) -- (-2.5,-1);
\draw (-3,1) -- (-2.5,-1);
\draw (-2,1) -- (-2.5,-1);
\draw (2.5, 1) ellipse (2 and 1);
\draw (1,1) node {$\bullet$};
\draw (2,1) node {$\bullet$};
\draw (3,1) node {$\bullet$};
\draw (3.5,1) node {...};
\draw (4.7,1.6) node {$A_{1}$};
\draw (2.5,-1) node {$\bullet$};
\draw (2.5,-1.5) node {$t_{1}$};
\draw (3,1) -- (2.5,-1);
\draw (2,1) -- (2.5,-1);
\draw (1,1) -- (2.5,-1);
\draw (-2.5,-1) -- (2.5,-1);
\draw (3.5,-1) node {...};
\draw (5,1) node {...};
\end{tikzpicture}
\end{minipage}
\caption{\em Graph $H_{1}$, a connected infinite graph where $\Delta(H_{1})=n+2\geq 4$.}
\label{Figure 4}
\end{figure}
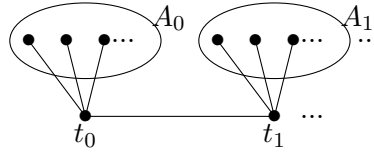
 
By the arguments in the proof of \cite[Claims 4.5, 4.6, 4.9, 4.10]{BMG2024} figured out by the authors, $\chi(G_{1})$, $\chi'(H_{1}), D(G),$ and $D'(G)$ do not exist. 
Thus, (1)-(10) fails in $\mathcal{M}(n)$.

\begin{claim}\label{Claim 6.2}
{\em Let $f: G\rightarrow D''(G)$ be a total distinguishing coloring of $G$, and $B = \{(\{t_n,x\},x): n\in \omega,  x\in A_n\}$. Let $g:B\rightarrow D''(G)\times D''(G)$ be the function that maps $(\{t,x\},x)$ to $(f(\{t,x\}), f(x))$.
Then the following holds:

\begin{enumerate}
    \item For each natural number $n$ and for each $x,y\in A_n$ such that $x\neq y$, we have $g(\{t_{n},x\},x)\neq g(\{t_{n},y\},y)$ 
    (i.e., $\{f(\{t_n, x\}),f(x)\}\neq \{f(\{t_n, y\}),f(y)\}$).
    \item $g[B]$ is an infinite set. 
\end{enumerate}
} 
\end{claim}

\begin{proof}
(1) is straightforward. Define, $N_{(c_{1},c_{2})}=\{(\{t_{n},x\},x)\in g^{-1}(c_{1},c_{2}): n\in\omega, x\in A_{n}\}$ for $(c_{1},c_{2})\in g[B]$. Clearly, $N_{(c_{1},c_{2})}$ is finite; otherwise, $\pi_{2}[N_{(c_{1},c_{2})}]$ will generate a partial choice function of $\mathcal{A}$ by (1), where $\pi_{2}$ is the function that maps $(\{t_{n},x\},x)$ to $x$. 
Now, $B=\bigcup_{(c_{1},c_{2})\in g[B]} N_{(c_{1},c_{2})}$. If $g[B]$ is finite, then $B$ is finite, which is a contradiction. Thus $g[B]$ is infinite. 
\end{proof}

We claim that $D''(G)$ is infinite (if it exists), and so (11) fails in $\mathcal{M}(n)$. If $D''(G)$ is finite, then $D''(G)\times D''(G)$ is finite. This contradicts Claim \ref{Claim 6.2}(2), since $g[B]\subseteq D''(G)\times D''(G)$. 
\end{proof}

\begin{thm}\label{Theorem 6.3}
{\em 
There is a model of $\mathsf{ZF}$ where
the statements ``If $G$ is a connected infinite graph where $\Delta(G)=3$, then $D(G)\leq 2$'' and ``Any connected infinite graph $G$ where $\Delta(G)=3$, has a distinguishing number'' fail. Thus, the result in Fact \ref{Fact 1.1}(9) fails in the model.} 
\end{thm}

\begin{proof}
Fix $n=2$, and assume the model $\mathcal{M}(2)$ as in the proof of Theorem \ref{Theorem 6.1} where
$\mathcal{A}=\{A_{i}:i\in \omega\}$ is a countably infinite family of $n$-element sets without a partial choice function. 
Pick two disjoint countably infinite sequences $T=\{t_{n}:n\in \omega\}$, and $R=\{r_{n}:n\in\omega\}$ which are also disjoint from $A=\bigcup_{i\in\omega}A_{i}$. 
Consider the following graph $G_{2}=(V_{G_{2}}, E_{G_{2}})$ where $\Delta(G_2)=3$ (see Figure \ref{Figure 5}):

\begin{alignat*}{3}
V_{G_{2}} &:= (\bigcup_{n\in\omega}A_n)\cup T \cup R,     \\
E_{G} &:= \big\{ \{r_n,r_{n+1}\}: n\in\omega\big\} &&\cup 
\{\{r_n,t_{n}\}: n\in\omega\big\}\cup
\big\{\{t_n,x\}: n\in\omega,x\in A_n \big\}. \\
\end{alignat*}

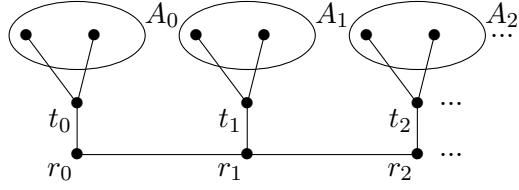
\begin{figure}[!ht]
\centering
\begin{minipage}{\textwidth}
\centering
\begin{tikzpicture}[scale=0.45]
\draw (-2.5, 1) ellipse (2 and 1);
\draw (-4,1) node {$\bullet$};
\draw (-2,1) node {$\bullet$};
\draw (-0,1.6) node {$A_{0}$};
\draw (-2.5,-1) node {$\bullet$};
\draw (-3,-1.5) node {$t_{0}$};
\draw (-2.5,-2.5) node {$\bullet$};
\draw (-3,-3) node {$r_{0}$};
\draw (-4,1) -- (-2.5,-1);
\draw (-2,1) -- (-2.5,-1);

\draw (2.5, 1) ellipse (2 and 1);
\draw (1,1) node {$\bullet$};
\draw (3,1) node {$\bullet$};
\draw (5,1.6) node {$A_{1}$};
\draw (2.5,-1) node {$\bullet$};
\draw (2,-1.5) node {$t_{1}$};
\draw (2.5,-2.5) node {$\bullet$};
\draw (2,-3) node {$r_{1}$};
\draw (3,1) -- (2.5,-1);
\draw (1,1) -- (2.5,-1);
\draw (-2.5,-2.5) -- (-2.5,-1);
\draw (2.5,-2.5) -- (2.5,-1);
\draw (-2.5,-2.5) -- (2.5,-2.5);

\draw (7.5, 1) ellipse (2 and 1);
\draw (6,1) node {$\bullet$};
\draw (8,1) node {$\bullet$};
\draw (10,1.6) node {$A_{2}$};
\draw (7.5,-1) node {$\bullet$};
\draw (7,-1.5) node {$t_{2}$};
\draw (7.5,-2.5) node {$\bullet$};
\draw (7,-3) node {$r_{2}$};
\draw (8,1) -- (7.5,-1);
\draw (6,1) -- (7.5,-1);
\draw (7.5,-2.5) -- (7.5,-1);
\draw (2.5,-2.5) -- (7.5,-2.5);

\draw (8.5,-2.5) node {...};
\draw (8.5,-1) node {...};
\draw (10,1) node {...};
\end{tikzpicture}
\end{minipage}
\caption{\em Graph $G_{2}$, a connected infinite graph where $\Delta(G_{2})=3$. In particular, the degree of $t_{i}$ and $r_{i+1}$ is 3 for each $i\in \omega$, while the degree of $r_{0}$ is $2$ and the degree of $x$ is $1$ for each $x\in A$.}
\label{Figure 5}
\end{figure}

We show that $G_{2}$ has no distinguishing number in $\mathcal{M}(2)$. Stawiski \cite[Claim 1]{Sta2023} proved 
that each $x\in V_{G_{2}}\backslash (\bigcup_{n\in\omega}A_n)$ is fixed by any automorphism of $G_{2}$ and for every $i\in\omega$ and $x\in A_{i}$, $Orb_{Aut(G_{2})}(x)=\{g(x): g\in Aut(G_{2})\}=A_i$.  
Assume that $D(G_{2})$ exists. Let  $f: V_{G_2} \to C$ be a distinguishing vertex coloring with $\vert C\vert=D(G_{2})$.
Similar to Claim \ref{claim 4.4}, $f[\bigcup_{n\in \omega}A_n]$ is infinite and Dedekind-finite. 
Consider a coloring $h: \bigcup_{n\in \omega}A_n\to f[\bigcup_{n\in \omega}A_n]\setminus\{c_0,c_{1}\}$ for some $c_{0}, c_{1}\in f[\bigcup_{n\in \omega}A_{n}]$, as in Claim \ref{claim 4.4}. Let,
$h(t)=c_{0}$ for all $t\in R\cup T$.    
Then, $h: V_{G_2} \rightarrow (f[\bigcup_{n\in \omega}A_{n}]\backslash \{c_{1}\})$ is a distinguishing vertex coloring of $G_2$ and 
			$f[\bigcup_{n\in \omega}A_n]\setminus\{c_1\} \prec  f[\bigcup_{n\in\omega} A_{n}] \preceq C$
			contradicts the fact that $\vert C\vert=D(G_{2})$.  
\end{proof}
\section{List-distinguishing chromatic numbers for infinite graphs}
\begin{defn}\label{Definition 7.1}
Let $T$ be a tree rooted at $v$. The {\em parent} of $z$ is the vertex immediately following $z$ on the unique path from $z$ to the root $v$. Two vertices are {\em siblings} if they have the same parent.    
\end{defn}

\begin{defn}\label{Definition 7.2}
Given an assignment $L = {L(v)}_{v\in V_{G}}$ of lists of available colors to the vertices of $G$, we say that $G$ is {\em properly L-distinguishable} if there is a distinguishing proper vertex coloring $f$ of $G$ such that $f(v) \in L(v)$ for all $v\in V_{G}$.  The {\em list-distinguishing chromatic number} of $G$, denoted by $\chi_{D_{L}}(G)$, is the least cardinal $\kappa$ (if it exists) such that $G$ is properly $L$-distinguishable for any list assignment $L$ with $\vert L(v)\vert = \kappa$ for all $v\in V_{G}$ (cf. \cite{FGHSW2013}). 
\end{defn}

\begin{remark}\label{Remark 7.3} 
Since all lists can be identical, $\chi_{D}(G) \leq \chi_{D_{L}}(G)$ for any graph $G$ in $\mathsf{ZFC}$. We construct a connected infinite graph $H$ where $\chi_{D_L}(H)\neq \chi_D(H)$, and $\Delta(H)=5$. Consider the graph $H$ given in Figure \ref{Figure 6}. Since $H$ is bipartite and has no non-trivial automorphisms, we have that $\chi_D(H) = \chi(H)= 2$. 

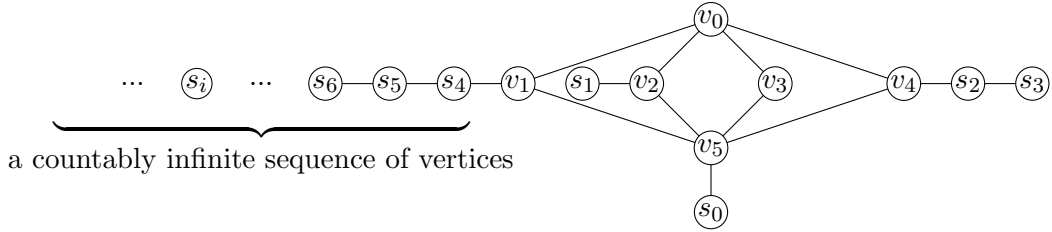
\begin{figure}[!ht]
\begin{minipage}{\textwidth}
\centering
\begin{tikzpicture}[scale=0.85]

\draw[black,] (0,-1) -- (-3,-2); 
\draw[black,] (0,-1) -- (-1,-2);
\draw[black,] (0,-1) -- (1,-2);
\draw[black,] (0,-1) -- (3,-2); 
\draw[black,] (3,-2) -- (4,-2); 
\draw[black,] (4,-2) -- (5,-2); 
\draw[black,] (-1,-2) -- (-2,-2); 

\draw[black,] (0,-3) -- (-3,-2);
\draw[black,] (0,-3) -- (-1,-2);
\draw[black,] (0,-3) -- (1,-2);
\draw[black,] (0,-3) -- (3,-2);
\draw[black,] (0,-3) -- (0,-4);

\draw[black,] (-6,-2) -- (-5,-2);
\draw[black,] (-5,-2) -- (-4,-2);
\draw[black,] (-4,-2) -- (-3,-2);

\node[circ] at (0,-1) {$v_0$};

\node[circ] at (-3,-2) {$v_1$};
\node[circ] at (-4,-2) {$s_{4}$};
\node[circ] at (-5,-2) {$s_{5}$};
\node[circ] at (-6,-2) {$s_{6}$};
\draw (-7,-2) node {$...$};
\node[circ] at (-8,-2) {$s_{i}$};
\draw (-9,-2) node {$...$};
\draw (-7,-2.5) node {$\underbrace{\phantom{n + 1\text{infinite sequence of vertices}}}$};

\draw (-7,-3.2) node {a countably infinite sequence of vertices};

\node[circ] at (-1,-2) {$v_2$};
\node[circ] at (-2,-2) {$s_1$};

\node[circ] at (1,-2) {$v_3$};
\node[circ] at (3,-2) {$v_4$};
\node[circ] at (4,-2) {$s_2$};
\node[circ] at (5,-2) {$s_3$};

\node[circ] at (0,-3) {$v_5$};
\node[circ] at (0,-4) {$s_0$};
\end{tikzpicture}    

\end{minipage}
\caption{\em A connected infinite graph $H$ where $\Delta(H)=5$ and $\chi_{D_L}(H)\neq \chi_D(H)=2$.}
\label{Figure 6}
\end{figure}
\end{remark}

Consider the following assignment of lists: $L(v_0) =\{1,2\}$, $L(v_1) = \{1,3\}$, $L(v_2) = \{2,3\}$, $L(v_3) = \{1,4\}$, $L(v_4) = \{2,4\}$, $L(v_5) =\{3,4\}$, and for $0\leq i \leq \omega$, $L(s_i)=\{5,6\}$. Then there is no proper coloring of $H$ from $\{L(v)\}_{v\in V_H}$, since $\{v_0,\dots, v_5\}$ do not have a proper coloring from the given lists. Thus, $\chi_{D_L}(H) \neq 2$. 


\begin{thm}\label{Theorem 7.4} {\em 
There is a model of $\mathsf{ZF}$ where
the statement
``If $G$ is a connected infinite graph where $\Delta(G)\geq 3$ is finite, then the list-distinguishing chromatic
number $\chi_{D_{L}}(G)$ is at most $2\Delta(G)-1$'' fails. Moreover, the statement holds under K\H{o}nig’s Lemma in $\mathsf{ZF}$.}     
\end{thm}

\begin{proof}
Let $\mathcal{A}=\{A_{n}:n\in \omega\}$ be a countably infinite family of $2$-element sets without a partial choice function in the model $\mathcal{M}(2)$ of $\mathsf{ZF}$ constructed in the proof of Theorem 6.1. Consider the graph $G$ from the proof of Theorem \ref{Theorem 4.1}, where $\Delta(G)= 4$. We show that $\chi_{D_{L}}(G)$ cannot be finite (if it exists) in  $\mathcal{M}(2)$. Assume that $\chi_{D_{L}}(G)$ exists. 
Let $L = {L(v)}_{v\in V_{G}}$ be an assignment of {\em identical} lists of available colors to the vertices of $G$ such that $\vert L(v)\vert = \chi_{D_{L}}(G)$. Then there is a distinguishing proper vertex coloring $f$ of $G$ such that $f(v) \in L(v)$ for all $v\in V_{G}$.
By the methods of Claim \ref{claim 4.4}(3),  $f[\bigcup_{n\in\omega}A_{n}]$ is infinite. Hence, $\chi_{D_{L}}(G)$ cannot be finite and the statement in the assertion fails. 

We show that the statement holds under K\H{o}nig’s Lemma. Since $G$ is locally finite and connected, $V_{G}$ is 
countably infinite (and hence well-orderable). Let $L=\{L(v)\}_{v\in V_G}$ be an assignment of lists with $|L(z)| = 2\Delta(G)-1$, for all $z\in V_G$. By K\H{o}nig's Lemma, $\bigcup_{z\in V_{G}}L(z)$ is countably infinite (and hence well-orderable) (cf. Definition \ref{Definition 2.5}). 
Let $v\in V_G$ be a vertex with degree $\Delta(G)$. By Observation \ref{Observation 3.2}, without invoking any form of choice, let $T$ be a $BFS$ spanning tree of $G$ rooted at $v$. Let,
\begin{itemize}
    \item $<$ denote the $BFS$ order of $T$, 
    \item $<'$ be a well-ordering of $\bigcup_{z\in V_{G}}L(z)$, and
    \item $<''$ be a well-ordering of $V_{G}$.
\end{itemize} 
For $z\in V_{G}$, $N(z)$ denotes the set of its neighbors in $G$, and $S(z)$ denotes the set of its siblings in $T$. The claim below holds in $\mathsf{ZF}$. 

\begin{claim}\label{claim 7.5}
{\em There are partial functions $f_x: V_{G} \to \bigcup_{z\in V_{G}}L(z)$ for each $x\in V_{G}$ so that the following holds: 
\begin{enumerate}
    \item[(i)] if $x < y$, then $f_x\subseteq f_y$,
    \item[(ii)] the domain of $f_{x}$, i.e., $dom(f_{x})$, includes all vertices till $x$ in the $BFS$ order.
    \item[(iii)] if $f=\bigcup_{x\in V_{G}} f_x$, then $f(w)\in L(w)$ for all $w\in V_{G}$ and $f$ is a proper vertex coloring of $G$. 
    \item[(iv)] With respect to the coloring $f$ mentioned in (iii), each vertex is colored differently from its siblings and from its parent in the $BFS$ spanning tree $T$.
\end{enumerate}  
}
\end{claim}

\begin{proof}
Let, $c_v=min_{<'}(L(v))$. If  $\{v_1,\dots, v_{\Delta(G)}\}$ are the neighbors of $v$ arranged in the $BFS$ order $<$, then let $c_{v_i}=min_{<'}(L(v_{i})\backslash \{c_{v_{j}}:1\leq j< i\})$ where $1\leq i\leq \Delta(G)$. Let $f_{v}=\{(v, c_{v})\}$, $f_{v_{1}}=f_{v}\cup \{(v_{1}, c_{v_{1}})\}$, and $f_{v_{i}}=f_{v_{i-1}}\cup \{(v_{i}, c_{v_{i}})\}$ for each $2\leq i\leq \Delta(G)$.
We proceed inductively in the $BFS$ order $<$ as follows: Let $x\in V_{G}$ be the $<$-minimal element for which no color has been assigned and $z<x$ be the immediate $BFS$ predecessor of $x$. We define $f_x$. Let us write,
$$
    S_<(x) = \{y\in S(x): y<x\} \text{ and } N_<(x) = \{y\in N(x): y<x\}.
$$

Then $f_z$ has already assigned colors to each member of $S_<(x) \cup N_<(x)$. Define,
$$
A_x= L(x)\setminus \Big(\{c_v\} \cup f_z[S_<(x)] \cup f_z[N_<(x)]\Big).
$$

\begin{figure}[!ht]
\centering
\begin{minipage}{\textwidth}
\centering
\begin{tikzpicture}[scale=0.6]
\draw (2, -1) ellipse (6 and 0.5);
\draw (9,-1) node {$N_{<}(x)$};

\draw (6,-3) node {$\bullet$};
\draw (6.5,-3) node {\tiny{$x$}};

\draw (2, -3) ellipse (3 and 0.25);
\draw (-1.8,-3) node {\tiny{$S_{<}(x)$}};

\draw (3,-4.5) rectangle (10,-3.8);
\draw (7.5,-5) node {\tiny{$A_x= L(x)\setminus \Big(\{c_v\} \cup f_z[S_<(x)] \cup f_z[N_<(x)]\Big)$}};
\draw (6,-4.1) node {$\bullet$};
\draw (8,-4.1) node {\tiny{$c=min_{<'}(A_{x})$}};
\draw [-triangle 60] (6,-3) to (6, -3.9);

\draw (3,-1) node {$\bullet$};
\draw (3.5,-1) node {\tiny{$t$}};
\draw (3,-1) to (6,-3);

\draw (3,-3) node {$\bullet$};
\draw (3.5,-3) node {\tiny{$z$}};
\draw (3,-1) to (3,-3);
\draw (2,-3) node {$...$};

\draw (0,-3) node {$\bullet$};
\draw (0.5,-3) node {\tiny{$y$}};
\draw (3,-1) to (0,-3);

\end{tikzpicture}
\end{minipage}
\caption{\em Defining $f_{x}$ when $A_{x}\neq \emptyset$.}
\label{Figure 7}
\end{figure}
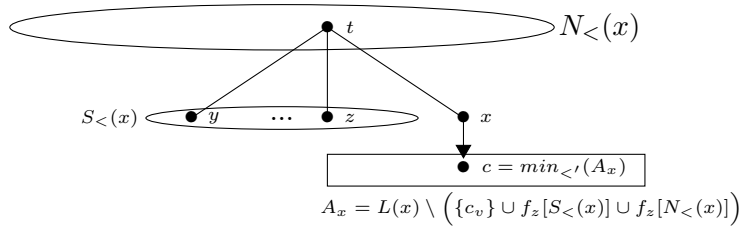

If $A_{x} \neq \emptyset$, then let $f_x = f_z \cup \{(x, c)\} $,
where $c=min_{<'}(A_{x})$ (see Figure \ref{Figure 7}). If $A_x=\emptyset$, then since $|N(x)| \leq \Delta(G)$ and $|S(x)|\leq \Delta(G)-2$, we must have used $|S_<(x)| + |N_<(x)| = 2\Delta(G) -2$ many colors from $L(x)$ to color the neighbors and siblings that come before $x$. 
Thus, the following holds: 

\begin{enumerate}
    \item[$(a)$] $c_v\in L(x)$,
    \item[$(b)$] $S_<(x)=S(x)$ and $N_<(x)=N(x)$,
    \item[$(c)$] $\vert S(x)\vert = \Delta(G) -2$ and $|N(x)| = \Delta(G)$,
    \item[$(d)$] $N(x) \cap S(x) = \emptyset$,
    \item[$(e)$] for all $y\in S(x) \cup N(x)$, we have $f_z(y)\in L(x)$.
\end{enumerate}
In this case, let $f_x = f_z\cup \{(x, c_v) \}$.
In either case, $f_x$ has (i) and (ii). 
Finally, set $f=\bigcup_{x\in V_{G}}f_x$. Clearly, $f$ satisfies (iii) and (iv), and if $f(x)=c_v$ where $x\neq v$, then $(a)-(e)$ must hold.
\end{proof}

Let $c_{v_1},\dots, c_{v_{\Delta(G)}}$ be as in the proof of Claim \ref{claim 7.5}. For any $x\in V_{G}$, we say $x$ has property ($\ast$) if the following holds:
\begin{center}
    {\em $x$ is colored by $c_{v}$,  while all its neighbors are colored by $c_{v_1},\dots, c_{v_{\Delta(G)}}$.}
\end{center}

Generalizing the algorithm of Imrich, Kalinowski, Pil{\'s}niak, and Shekarriz \cite[Theorem 3]{IKPS2017} in the context of list coloring as well as in $\mathsf{ZF}$, we try to keep $v$ as the only vertex of $G$ with property ($\ast$) (cf. \cite[Theorem 3.1]{BMG} for finite graphs).
Then $f$ is a proper distinguishing coloring. 
However, suppose there is a vertex other than $v$, say $x$, with the property $(\ast)$ and is the $<$-minimal such vertex. Then $(a)-(e)$ in Claim \ref{claim 7.5} are satisfied for $x$. We modify $f$ so that $x$ no longer has the property $(\ast)$, while it remains a proper coloring, by analyzing the following cases:\footnote{We note that \textsc{Case 1} and {\em Case A} of \textsc{Subcase 3.1} arise when the lists are nonidentical. Hence, they didn't appear in the algorithm of Imrich, Kalinowski, Pil{\'s}niak and Shekarriz \cite[Theorem 3]{IKPS2017}. Subcase 3.2 is also slightly different in our case.} 

\noindent \textsc{Case 1:} There is $w\in S(x)$ such that $|N(w)| \neq |N(x)|$ (and let $w$ be the $<''$-minimal such element). We color $x$ with $f(w)$,
which can be done by $(e)$. Then there is no color-preserving automorphism of $G$, which fixes $v$ and maps $x$ to $w$. By $(d)$, the new coloring is proper. Moreover, $x$ no longer has the property $(\ast)$.

\noindent \textsc{Case 2:} For every $z\in S(x)$, we have $\vert N(z)\vert = \vert N(x)\vert$, and there is a $w\in S(x)$ such that $N(w)\neq N(x)$ (and let $w$ be the $<''$-minimal such element). By $(c)$, we have $N(x)\backslash N(w)\neq\emptyset$. Fix $y\in N(x)\backslash N(w)$. By $(e)$, we have $f(w), f(y)\in L(x)$ .
We color $x$ with $f(w)$. The rest follows the arguments in \textsc{Case 1}.

\noindent  \textsc{Case 3: } For every $w\in S(x)$, we have $N(w) = N(x)$.

\textsc{Subcase 3.1:} There is a vertex $z\in N(x)$ with no sibling or parent colored with $c_v$ (and let $z$ be the $<''$-minimal such element). We need to consider two different cases.
\begin{enumerate}
    \item[]\textit{Case A:}
    If $c_v\not\in L(z)$, then by $(c)$ at most $(\Delta(G)-1) + (\Delta(G)-2) = 2\Delta(G)-3$ number of colors can appear in the list $L(z)$ that are used for coloring $N(z)\cup S(z)$. Hence, there exists $c\in L(z)$ such that $c\neq f(z)$, which also differs from all the colors assigned to $N(z)\cup S(z)$. Assume $c$ is the $<'$-minimal such element of $L(z)$.
    We color $z$ with $c$ and $x$ with $f(z)$.
    Then $x$ no longer has the property $(\ast)$ as no member of $N(v)$ has color $c$, and this modification keeps the coloring proper.

    Moreover, $z$ does not have the property 
    $(\ast)$ as $z$ is colored by $c$ but $c\neq c_{v}$ since $c\in L(z)$ and $c_{v}\not\in L(z)$.
    
   \item[]  \textit{Case B:} If $c_v\in L(z)$, then color $x$ with $f(z)$ and $z$ with $c_{v}$, which can be done by $(e)$. This modification keeps the coloring proper and $x$ does not have the property $(\ast)$.

   We show that $z$ does not have the property $(\ast)$. For the sake of contradiction, assume that $z$ has the property $(\ast)$. We show that there exists $w\in N(z)$ such that $w$ is not colored by any of $c_{v_{1}},...,c_{v_{\Delta(G)}}$ to obtain a contradiction.  Since $\Delta(G)\geq 3$, we have $\vert S(x) \vert=\Delta(G)-2\geq 1$. Pick any $w\in S(x)$. 
   Since $x$ had the property $(\ast)$ initially, the members of $N(x)$ were colored by the colors of $\{c_{v_{1}},...,c_{v_{\Delta(G)}}\}$. After recoloring, the members of $N(x)$ were colored by the colors of the set $C(z)=(\{c_{v_{1}},...,c_{v_{\Delta(G)}}\}\cup \{c_{v}\})\backslash \{f(z)\}$ where $f(z)\in \{c_{v_{1}},...,c_{v_{\Delta(G)}}\}$. 
   Since $w$ is adjacent to all members of $N(x)$, $w$ is not colored with the colors of $C(z)$ since the modified coloring is proper. Moreover, $w$ is not colored by $f(z)$. Otherwise, $z$ has two neighbors $w$ and $x$ with the same color $f(z)$ which implies that $z$ does not have ($\ast$), and we obtain a contradiction. Thus, $w$ is not colored by any of $c_{v_{1}},...,c_{v_{\Delta(G)}}$.
   
    \end{enumerate}

\textsc{Subcase 3.2:} Each $y \in N(x)$ has a sibling or parent colored with $c_v$. By $(b)$, each $y\in N(x)$ and its siblings must have come before $x$ in the $BFS$ ordering $<$. Hence, their parent cannot have the property ($\ast$), unless it is the root $v$ (since $x$ is assumed to be the $<$-minimal  vertex with the property $(\ast)$).

\begin{figure}[!ht]
\centering
\begin{minipage}{\textwidth}
\centering
\begin{tikzpicture}[scale=0.6]
\draw (2.5, -1) ellipse (6 and 0.5);
\draw (-6,-1) node {\tiny{$N_{<}(x)=N(x)=N(w)$}};

\draw (6,-3) node {$\bullet$};
\draw (8.5,-3) node {\tiny{$x$ (colored with $c_{v}$)}};

\draw (2, -3) ellipse (3 and 0.25);
\draw (-2.5,-3) node {\tiny{$S_{<}(x)=S(x)$}};

\draw (8,-0.5) rectangle (19.6,0.5);
\draw (8,-0.5) rectangle (13.3,0.5);
\draw (8,-0.5) rectangle (16.7,0.5);
\draw (18,-0.9) node {\tiny{$L(z)$}};
\draw (8.7,0) node {$\bullet$};
\draw (11,0) node {\tiny{$c=min_{<'}(L'(z))$}};
\draw (18.2,0) node {$\bullet$};
\draw (14.5,0) node {\tiny{$B(z)$}};
\draw (19,0) node {\tiny{$f(z)$}};

\draw [dashed,-triangle 60] (3,-1) to (8.5, 0);
\draw [dashed,-triangle 60] (6,-3) to (18, -0.2);

\draw (3,-1) node {$\bullet$};
\draw (6,-1) node {$\bullet$};
\draw (0,-1) node {$\bullet$};
\draw (2.5,-1) node {\tiny{$z$}};
\draw (-1,-1) node {$...$};
\draw (3,-1) to (6,-3);
\draw (6,-1) to (6,-3);
\draw (0,-1) to (6,-3);
\draw (3,-1) to (3,0);

\draw (3,0) node {$\bullet$};
\draw (3.5,0) node {\tiny{$p$}};

\draw (3,-3) node {$\bullet$};
\draw (3.5,-3) node {\tiny{$w$}};
\draw (6,-1) to (3,-3);
\draw (0,-1) to (3,-3);
\draw (3,-1) to (3,-3);

\draw (2,-3) node {$...$};

\draw (0,-3) node {$\bullet$};
\draw (3,-1) to (0,-3);
\draw (6,-1) to (0,-3);
\draw (0,-1) to (0,-3);

\end{tikzpicture}
\end{minipage}
\caption{\em In Subcase 3.2, $N(w) = N(x)$ for every $w\in S(x)$, and 
$z \in N(x)$ has a sibling or parent colored with $c_v$. Moreover, $\vert N(x)\vert=\Delta(G)$ and $S(x)=\Delta(G)-2$.}
\label{Figure 8}
\end{figure}
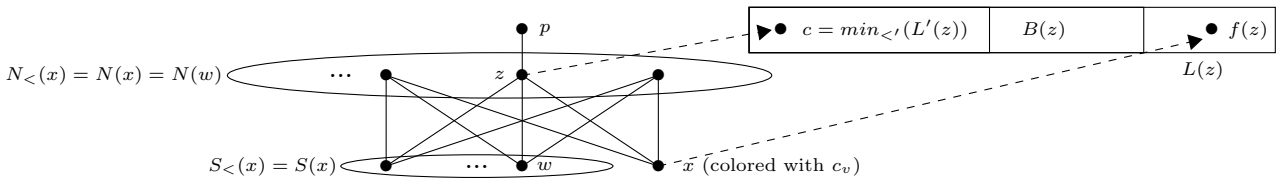

\begin{claim}
    {\em There exists a $z\in N(x)$ such that $z$ is not a child of $v$. }
\end{claim}

\begin{proof}
Otherwise, $z$ is a child of $v$ for all $z\in N(x)$. We recall the following:

\begin{itemize}
    \item $\vert S(x)\vert = \Delta(G) -2$ and $\vert N(x)\vert = \Delta(G)$ (see $(c)$).
    \item For every $w\in S(x)$, we assumed that $N(w) = N(x)$ in Case 3. Thus, $\vert N(w)\vert = \Delta(G)$.
    \item For every $z\in N(x)$, $z$ is a child of $v$ and so $N(z)=\{x\}\cup \{v\}\cup S(x)$ and $\vert N(z)\vert = \Delta(G)$.
\end{itemize}

Consequently, $G$ is isomorphic to a complete bipartite graph $K_{\Delta(G),\Delta(G)}$ where the bipartitions are $N(x)$ and $\{x\}\cup \{v\} \cup S(x)$.
Since $\Delta(G)$ is assumed to be finite, this contradicts the assumption that $G$ is infinite. 
\end{proof}

Without loss of generality, let $z$ be the $<''$-minimal such element. Let $B(z)$ denotes the set of colors assigned to the vertices of $N(z)\cup S(z)$.

\begin{claim}\label{claim 7.6}
{\em The set $L'(z)=L(z)\backslash \{f(z)\cup B(z)\}$ is non-empty.}
\end{claim}
\begin{proof}
Let $p$ be the parent of $z$ in $T$. Then $N(z) = S(x) \cup \{p\}\cup \{x\}$, $\vert S(z)\vert \leq \Delta(G)-2$, and $\vert N(z)\vert = \vert S(x)\cup \{p\}\cup \{x\}\vert=\Delta(G)$. The color $c_v$ was used once for $x$ and once among the vertices of $S(z) \cup \{p\}$. Thus, $\vert B(z)\vert\leq 2\Delta(G)-3$. Since $\vert L(z)\vert= 2\Delta(G)-1$, the rest follows.
\end{proof}

Define $c=min_{<'}(L'(z))$. Using $(e)$, we can color $x$ with $f(z)$ and $z$ with $c$ (see Figure \ref{Figure 8}).  
We note that $z$ does not get property $(
\ast)$. This is because $c\neq c_{v}$ since $c_{v} \in B(z)$.
Similarly, as in the previous cases, the new coloring remains a proper coloring, and $v$ is the only $<$-predecessor of $x$ having the property $(\ast)$, and all of them are fixed as long as $v$ is fixed.

By iterating this procedure at most countably-infinitely many times (since $V_{G}$ is countably infinite), we can recolor every vertex with the property $(\ast)$ and obtain a distinguishing proper coloring $g$ of $G$ such that $g(v)\in L(v)$ for all $v\in V_{G}$.
Hence, $G$ is properly $L$-distinguishable.
\end{proof}
\section{Remarks and Questions}

\begin{remark}\label{Remark 8.1}
The results mentioned in Fact \ref{Fact 1.1}((1)-(6) and (9)) follow from K\H{o}nig's Lemma in $\mathsf{ZF}$.  Following the arguments of Theorem \ref{Theorem 3.3}(2) and \cite[Theorems 3,7]{BR1976}, the statements (1) and (2) in Fact \ref{Fact 1.1} hold in $\mathsf{ZF}$ if the vertex set $V_{G}$ of $G$ is countably infinite. One may modify the algorithms mentioned in the references in Fact \ref{Fact 1.1}((3)-(6) and (9)) to see that those results hold in $\mathsf{ZF}$ + K\H{o}nig's Lemma, if $V_{G}$ is countably infinite. The rest follows from the fact that any locally finite, connected graph is countably infinite under K\H{o}nig's Lemma.
\end{remark}

\begin{question}\label{Question 8.2}
		Are the following statements equivalent to $\mathsf{AC}$ (without assuming that the sets of colors can be well-ordered)?
		\begin{enumerate}
               \item Every connected infinite graph has a chromatic index.
               \item Every connected infinite graph has a distinguishing number.
               \item Every connected infinite graph has a distinguishing index.
                \item Every connected infinite graph has a distinguishing chromatic number.   
			\item Every connected infinite graph has a distinguishing chromatic index.
			\item Every connected infinite graph has a total chromatic number.
			\item Every connected infinite graph has a total distinguishing chromatic number.
                \item Every connected infinite graph has a neighbor-distinguishing number.
                \item Every connected infinite graph has a neighbor-distinguishing index.
		\end{enumerate}
   \end{question}
\section{Appendix}
In this section, we list the complete definitions of the graph coloring parameters.

\begin{defn}\label{Definition 9.1}
A graph $G=(V_{G}, E_{G})$ consists of a set $V_{G}$ of vertices and a set $E_{G}\subseteq [V_{G}]^{2}$ of edges (i.e., $E_{G}$ is a subset of the set of all two-element subsets of $V_{G}$). The {\em degree} of a vertex $v\in V_{G}$, denoted by $deg(v)$, is the number of edges emerging from $v$. We denote by $\Delta(G)$ the maximum degree of $G$. Fix $n\in \omega$. A {\em path of length $n$} is a one-to-one finite sequence $\{x_{i}\}_{0\leq i \leq n}$ of vertices such that for each $i < n$, $\{x_{i}, x_{i+1}\} \in E_{G}$; such a path joins $x_{0}$ to $x_{n}$.  
A {\em ray} is an infinite graph with vertex set $V=\{v_{1},v_{2},...\}$ such that there is an edge between $v_{i}$ and $v_{i+1}$ for each natural number $i$.
The graph $G$ is {\em locally finite} if every vertex of $G$ has a finite degree and it is {\em connected} if any two vertices are joined by a path of finite length. {\em K\H{o}nig’s Lemma} states that every infinite locally finite connected graph has a ray. The cycle graph on $n$ vertices for any natural number $n\geq 3$, is denoted by $C_{n}$. 
A {\em tree} is a connected graph with no cycles, and a {\em spanning tree} $T$ of $G$ is a subgraph that is a tree and contains all the vertices of $G$.
The automorphism group of $G$, denoted by $Aut(G)$, is the group consisting of automorphisms of $G$ with composition as the operation. Let $\tau$ be a group acting on a set $S$ and let $a\in S$. The orbit of $a$, denoted by $Orb_{\tau}(a)$, is the set $\{\phi(a) : \phi \in \tau\}$.
Two vertices $x, y \in V_{G}$ are {\em adjacent vertices} if $\{x, y\} \in E_{G}$, and two edges $e,f\in E_{G}$ are {\em adjacent edges} if they share a common vertex.		
  \begin{enumerate}
			\item A {\em proper vertex coloring} of $G$ with a color set $C$ is a mapping $f:V_{G}\rightarrow C$ such that for every $\{x,y\}\in E_{G}$, $f(x)\not= f(y)$. 
                \item A {\em proper edge coloring} of $G$ with a color set $C$ is a mapping $f:E_{G}\rightarrow C$ such that for any two adjacent edges $e_{1}$ and $e_{2}$, $f(e_{1})\not= f(e_{2})$.  
                \item A {\em total coloring} of $G$ with a color set $C$ is a mapping $f:V_{G}\cup E_{G}\rightarrow C$ such that for every $\{x,y\}\in E_{G}$, $f(x), f(\{x,y\})$ and $f(y)$ are all different and for any two adjacent edges $e_{1}$ and $e_{2}$, $f(e_{1})\not= f(e_{2})$.
			
		\item An {\em odd proper vertex coloring} of $G$ is a proper vertex coloring $f$ such that for each non-isolated vertex $x \in V_{G}$, there exists a color $c$ such that $\vert f^{-1}(c) \cap N(x)\vert$ is an odd integer where $N(v)=\{u \in V_{G} : \{u,v\} \in E_{G}\}$ is the open neighborhood of $v$ for any $v\in V_{G}$ (cf. \cite[Introduction]{CPS2022}).

            \item An automorphism of $G$ is a bijection $\phi: V_{G}\rightarrow V_{G}$ such that $\{u,v\}\in E_{G}$ if and only if $\{\phi(u),\phi(v)\}\in E_{G}$. 
			Let $f$ be an assignment of colors to either vertices or edges of $G$. 
			We say that an automorphism $\phi$ of $G$ {\em preserves $f$} if each vertex (edge) of $G$ is mapped to a vertex (edge) of the same color.
			We say that $f$ is a {\em distinguishing coloring} if the only automorphism that preserves $f$ is the identity and $f$ is a {\em distinguishing proper vertex coloring} if $f$ is a distinguishing coloring as well as a proper vertex coloring. Similarly, we can define a {\em distinguishing proper edge coloring} (cf. \cite[Introduction]{IKPS2017}).

            \item A {\em total distinguishing coloring} of $G$ with a color set $C$ is a mapping
            $g:V_{G}\cup E_{G}\rightarrow C$ such that $g$ is only preserved by the trivial automorphism. Moreover, $g$ is a {\em properly total distinguishing coloring} if $g$ is total distinguishing and it is a total coloring (see \cite{KPW2016}). 
            
			\item Let $\vert C\vert =\kappa$. We say $G$ is {\em $\kappa$-proper vertex colorable ($\kappa$-odd proper vertex colorable, $\kappa$-total colorable,  $\kappa$-proper edge colorable, $\kappa$-distinguishing vertex colorable, $\kappa$-distinguishing proper vertex colorable, $\kappa$-distinguishing edge colorable, $\kappa$-distinguishing proper edge colorable, $\kappa$-distinguishing total colorable, $\kappa$-properly distinguishing total colorable)}  if there is a proper vertex coloring (odd proper vertex coloring, total coloring, proper edge coloring, distinguishing vertex coloring, distinguishing proper vertex coloring, distinguishing edge coloring, distinguishing proper edge coloring, total distinguishing coloring, properly total distinguishing coloring) $f$ with a color set $C$.
            \item The least cardinal $\kappa$ for which $G$ is $\kappa$-proper vertex colorable ($\kappa$-total colorable, $\kappa$-odd proper vertex colorable, $\kappa$-proper edge colorable,  $\kappa$-distinguishing vertex colorable, $\kappa$-distinguishing proper vertex colorable, $\kappa$-distinguishing edge colorable, $\kappa$-distinguishing proper edge colorable, $\kappa$-distinguishing total colorable, $\kappa$-properly distinguishing total colorable) if it exists, is the {\em chromatic number $(\chi(G))$} ({\em total chromatic number $(\chi''(G))$, odd chromatic number $(\chi_{o}(G))$, chromatic index $(\chi'(G))$, distinguishing number $(D(G))$, distinguishing chromatic number $(\chi_{D}(G))$, distinguishing index $(D'(G))$, distinguishing chromatic index $(\chi'_{D}(G))$, total distinguishing number $(D''(G))$, total distinguishing chromatic number ($\chi''_{D}(G)$})) of $G$.

        \item For a $\kappa$-proper edge coloring
$f$ of $G$, $C_{f}(v)$ denotes the set of colors assigned to the edges incident with a vertex $v$. The coloring $f$ is called {\em $\kappa$-neighbor-distinguishing proper edge coloring} if $C_{f}(u) \neq C_{f}(v)$ for any $\{u,v\}\in E_{G}$ (cf. \cite[Introduction]{HCWH2023}).
For a $\kappa$-proper vertex coloring
$f$ of $G$, $N_{f}(v)$ denotes the set of colors of the neighbors of $v$. The $\kappa$-proper vertex coloring $f$ is called {\em neighbor-distinguishing} if $N_{f}(u) \neq N_{f}(v)$ for any $\{u,v\}\in E_{G}$. 
The {\em neighbor-distinguishing index $(\chi'_{N}(G))$} ({\em neighbor-distinguishing number $(\chi_{N}(G))$}) of $G$ is the least cardinal $\kappa$  (if it exists) such that $G$ has a $\kappa$-neighbor-distinguishing proper edge coloring ($\kappa$-neighbor-distinguishing proper vertex coloring). 
\end{enumerate}
\end{defn}

\section{Acknowledgement} The authors are very thankful to the three anonymous referees for reading the manuscript in detail and for providing several comments
and suggestions that improved the quality of the paper.

\end{document}